\documentclass{elsarticle}

\usepackage{latexsym,enumerate}
\usepackage{amsmath,amsthm,amsopn,amstext,amscd,amsfonts,amssymb}
\usepackage[ansinew]{inputenc}
\usepackage{verbatim}
\usepackage{pstricks}
\usepackage{wasysym}
\usepackage[all]{xy}
\usepackage{epsfig} 
\usepackage{graphicx,psfrag} 
\usepackage{subfigure}

\newtheorem{theorem}{Theorem}[section]
\newtheorem{lemma}[theorem]{Lemma}
\newtheorem{proposition}[theorem]{Proposition}
\newtheorem{corollary}[theorem]{Corollary}

\newtheorem{remark}[theorem]{Remark}
\newtheorem{algorithm}[theorem]{Algorithm}

\DeclareMathOperator{\diam}{diam}

\DeclareMathOperator{\Deg}{Deg}

\newcommand{\spb}[1]{\smallskip}
\newcommand{\mpb}[1]{\medskip}
\newcommand{\bpb}[1]{\bigskip}

\renewcommand{\d}{\delta}

\newcommand{\D}{\Delta}

\journal{}

\begin{document}

\begin{frontmatter}

\title{Alliance polynomial of regular graphs}

\author[a]{Walter Carballosa}
\address[a]{National council of science and technology (CONACYT) $\&$ Autonomous University of Zacatecas,
Paseo la Bufa, int. Calzada Solidaridad, 98060 Zacatecas, ZAC, Mexico}
\ead{waltercarb@gmail.com} 

\author[b]{Jos\'e M. Rodr{\'\i}guez}
\address[b]{Department of Mathematics, Universidad Carlos III de Madrid, Av. de la Universidad 30, 28911 Legan\'es, Madrid, Spain}
\ead{jomaro@math.uc3m.es}

\author[d]{Jos\'e M. Sigarreta}
\address[d]{Faculty of Mathematics, Autonomous University of Guerrero, Carlos E. Adame 5, Col. La Garita, Acapulco, Guerrero, Mexico}
\ead{jsmathguerrero@gmail.com}

\author[e]{Yadira Torres-Nu\~nez}
\address[e]{Humboldt International University, 4000 West Flagler Street, 33134, Miami, Fl., USA}
\ead{yadiratn83@gmail.com}

\begin{abstract}
The alliance polynomial of a graph $G$ with order $n$ and maximum degree $\D$ is the polynomial $A(G; x) = \sum_{k=-\D}^{\D} A_{k}(G) \, x^{n+k}$, where $A_{k}(G)$ is the number of exact defensive $k$-alliances in $G$. We obtain some properties of $A(G; x)$ and its coefficients for regular graphs. In particular, we characterize the degree of regular graphs by the number of non-zero coefficients of their alliance polynomial.
Besides, we prove that the family of alliance polynomials of $\D$-regular graphs with small degree is a very special one, since it does not contain alliance polynomials of graphs which are not $\D$-regular.
By using this last result and
direct computation we find that the alliance polynomial determines uniquely each cubic graph of order less than or equal to $10$.
\end{abstract}


\end{frontmatter}

%
%
%
%
%
%
%
%
%
%
%
%
%
%
%

\section{Introduction.}
Graph polynomials have been widely study since George D. Birkhoff introduced the chromatic polynomial (1912) in an attempt to prove the four color theorem (see \cite{Bi}).
Although the original motivation for the study of this invariant (chromatic number) is still important, much of the current interest, for example, in the Tutte polynomial is not related to any of its applications.
In particular, graph polynomials are considered interesting when they encode much information about the underlying graph.
Some parameters of a graph $G$ allow one to define polynomials on the graph $G$, for instance,
the parameters associated to matching sets \cite{F,GG}, independent sets \cite{BDN,GH,HL}, domination sets \cite{AAP,AP}, chromatic numbers \cite{Bi,R,T} and many others.
In recent years there has been an increase in the number of
papers introducing new properties of graph polynomials \cite{KM,MR,MRB}.

The study of defensive alliances in graphs, together with a variety of other kinds of alliances, was introduced in \cite{KHH}. In the cited paper authors initiated the study of the mathematical properties of alliances. In particular, several bounds on the defensive alliance number were given. The particular case of global defensive alliance (generalization of dominating sets) was investigated in \cite{FLH3,HHH,SRV2} where several bounds on the global defensive alliance number were obtained. Several tight bounds on different types of alliance numbers of a graph were obtained in \cite{KHH,SRV2,S,SGBR}, namely, (global) defensive alliance number, (global) offensive alliance number and (global) dual alliance number. Moreover, alliances, as a graph-theoretic concept, have recently attracted a great deal of attention due to some interesting applications in a variety of areas, including quantitative analysis of secondary RNA structures \cite{HKSZ} and national defense \cite{P}.
Besides, defensive alliances could work as a mathematical model of web communities. Adopting the
definition of a Web community proposed recently in \cite{FLG}, “a Web community is a set of web pages
having more hyperlinks (in either direction) to members of the set than to non-members”.
In \cite{CHRT}, the authors use the cardinality of the strong defensive alliances by define the strong alliance polynomial of a graph.
The exact defensive alliance was defined in \cite{C}.
The alliance polynomial was introduced in \cite{To} and have been studied in \cite{CRST,CRSTx}.

Throughout this paper, $G=(V,E)$ denotes a simple graph (not necessarily connected) of order $|V|=n$ and size $|E|=m$.
We denote two adjacent vertices $u$ and $v$ by
$u\sim v$. For a nonempty set $X\subseteq V$, and a vertex $v\in V$,
 $N_X(v)$ denotes the set of neighbors  $v$ has in $X$:
$N_X(v):=\{u\in X: u\sim v\},$ and the degree of $v$ in $ X$ will be
denoted by $\delta_{X}(v)=|N_{X}(v)|$. We denote by $\d$ and $\D$ the minimum and maximum degree of $G$, respectively.
If $G$ is a regular graph with degree $\D$ we say that $G$ is $\D$-regular.
The subgraph induced by
$S\subset V$ will be denoted by  $\langle S\rangle $ and the
complement of the set $S\in V$ will be denoted by $\overline{S}$.

A nonempty set $S\subseteq V$ with $\langle S\rangle$ connected is a \emph{defensive  $k$-alliance} in
$G=(V,E)$,  $k\in [-\Delta,\Delta]\cap \mathbb{Z}$, if for
every $ v\in S$,
\begin{equation}\label{cond-A-Defensiva} \delta _S(v)\ge \delta_{\overline{S}}(v)+k.\end{equation}
A vertex $v\in S$ is said to be $k$-\emph{satisfied} by the set $S$,
if \eqref{cond-A-Defensiva} holds.

We consider the value of $k$ in the set of integers $\mathcal{K}:= [-\Delta,\Delta]\cap \mathbb{Z}$.
It may be that there are some values of $k \in\mathcal K$ such that there do not exist. For instance, for $k\ge 2$ there are no defensive $k$-alliances in the star graph $S_n$ with $n$ vertices.
Besides, if $G$ is connected then $V(G)$ is a defensive $\d$-alliance.
Notice that for any $S$ with $\langle S\rangle$ connected there exists some $k\in\mathcal{K}$ such that it is a defensive $k$-alliance in $G$.

We define, for $S\subseteq V(G)$ inducing a connected subgraph $\langle S\rangle$,
\begin{equation}\label{eq:k_exact}
    k_{S}:=\displaystyle\max_{} \{k\in\mathcal{K} \,:\,  S \text{ is a defensive $k$-alliance}\}.
\end{equation}
We say that  $k_{S}$ is the \emph{exact index of alliance of} $S$, or also, $S$ is an \emph{exact defensive $k_{S}$-alliance} in $G$.
The exact index of alliance of $S$ in $G$ is also $k_{S}= \displaystyle\min_{v\in S} \{\d_{S}(v) - \d_{\overline{S}}(v)\}$.

The \emph{alliance polynomial} of a graph $G$ is defined as an ordinary generating function for the number of exact defensive $k$-alliances:
\begin{equation}\label{eq:Poly1}
    A(G;x)= \displaystyle\sum_{k\in \mathcal{K}} A_{k}(G) x^{n+k}, \ \text{ with } A_{k}(G) \text{ the number of exact defensive } k \text{-alliances in } G.
\end{equation}

There are several d.p.-equivalent definitions for an alliance polynomial. For instance, another natural definition for alliance polynomials of $G$ could be $A^\star(G;x)= \sum_{k\in \mathcal{K}} A_{k}(G) x^{k}$ (\emph{i.e.}, $A(G;x)= x^n\, A^\star(G;x)$). We may also define it by $A^\dag(G;x)=\sum_{k\in \mathcal{K}} \big(A_k(G)+1\big) x^k$. Note that, $A^\star(G;x)$ and $A^\dag(G;x)$ are Laurent polynomials, and that the polynomial $A^\dag(G;x)$ does not satisfy any easy variant of Corollary \ref{c:CubUnim}. 
Hence it is convenient to keep the previous choice, although each result in this paper has an analogous for $A^\star(G;x)$.

The following procedure allows to compute the alliance polynomial of a graph $G$ with order $n$, see \cite[Algorithm 2.1]{CRST}. This will be used in Section \ref{sect3} in order to compute alliance polynomials.
Let us consider $W=\{S_1,\dots,S_{2^n-1}\}$ the collection of nonempty subsets of $V$.
\begin{algorithm}\label{algorithm}
${}$\newline
Input: adjacency matrix of $G$.
\newline
Output: alliance polynomial of $G$.
\smallskip\newline
The algorithm starts with $A(G;x)=0$ and continues with the following steps, for $1 \le j \le 2^n-1$.
\begin{enumerate}
  \item {If $\langle S_j\rangle$ is a connected subgraph, then go to step (2), else replace $j$ by $j+1$ and apply this step again.}
  \item {Compute $k_{S_j}$, and add one term $x^{n+k_{S_j}}$ to $A(G;x)$.}
  \item {If $j < 2^n-1$, then replace $j$ by $j+1$ and apply step (1) again. If $j = 2^n-1$, then the algorithm stops.}
\end{enumerate}
\end{algorithm}

The main aim of this work is to obtain further results about the alliance polynomial of regular graphs (graphs with all vertices with the same degree), since they are a very interesting class of graphs with many applications (see, e.g., \cite{B,Cu,DM,RST}).
In this paper we study the alliance polynomials of regular graphs and their coefficients, see Section \ref{sect:reg}.
In Section \ref{sect:con_reg} we focus on the alliance polynomials of connected regular graphs; besides, we prove that the family of alliance polynomials of connected $\D$-regular graphs with small degree is a very special one, since it does not contain alliance polynomials of graphs which are not connected $\D$-regular (see Theorems \ref{t:Reg0-3} and \ref{t:Reg4-6}).
Finally, by using Theorem \ref{t:Reg0-3} and
direct computation we find that the alliance polynomial determines uniquely each cubic graph of order less than or equal to $10$.

\section{Alliance polynomials for regular graphs}\label{sect:reg}
In this section we deal with regular graphs, in particular, we obtain some properties of the alliance polynomial of regular graph and its coefficients.
Below, a quick reminder of some previous results for general graphs (not necessarily regular) which will be useful, see \cite{CRST}. We denote by $\Deg(p)$ the degree of the polynomial $p$.

\begin{theorem}\label{p:AlliPoly}
   Let $G$ be any graph. Then $A(G;x)$ satisfies the following properties:

   \begin{enumerate}[i)]
     \item {$A(G;x)$ does not have zeros in the interval $(0,\infty)$.}

     \item {$A(G;1) < 2^{n}$, and it is the number of connected induced subgraphs in $G$.}

     \item {If $G$ has at least an edge and its degree sequence has exactly $r$ different values, then $A(G;x)$ has at least $r + 1$ terms.}

     \item { $A(G;x)$ is a symmetric polynomial (i.e., an even or an odd function of $x$) if and only if the degree sequence of $G$ has either all values odd or all even.}

     \item {$A_{-\D}(G)$ and $A_{-\D + 1}(G)$ are the number of vertices in $G$ with degree $\D$ and $\D-1$, respectively.}

     \item {$A_{\D}(G)$ is equal to the number of connected components in $G$ which are $\D$-regular.}


     \item {$n + \d \leq \Deg(A(G;x)) \leq n + \D$.}
   \end{enumerate}
\end{theorem}

As usual, by \emph{cycle} we mean a simple closed curve, i.e., a path with different vertices,
unless the last one, which is equal to the first vertex.
The following lemma is a well known result of graph theory.

\begin{lemma}\label{l:degreem}
If $r\ge 2$ is a natural number and $G$ is any graph with $\d(v) \ge r$ for every
$v \in V(G)$, then there exists a cycle $\eta$ in $G$ with $L(\eta) \ge r+1$.
\end{lemma}

We show now some results about the alliance polynomial of regular graphs and their coefficients.
If $G$ is a graph and $v\in V(G)$, we denote by $G \setminus \{v\}$ the subgraph obtained by removing from $G$ the vertex $v$ and the edges incident to $v$.
We say that $v$ is a \emph{cut vertex} if
$G \setminus \{v\}$ has more connected components than $G$.
Besides, if $p$ is a polynomial we denote by $\Deg_{min}(p)$ the minimum degree of their non-zero coefficients.

\begin{theorem}\label{t:PropRegular}
   For any $\D$-regular graph $G$, its alliance polynomial $A(G;x)$ satisfies the following properties:

   \begin{enumerate}[i)]
     \item {$A_{-\D+2i}(G)$ is the number of connected induced subgraphs of $G$ with minimum degree $i$ $(0\le i\le\D)$.}

     \item {$\Deg_{\min}\big(A(G;x)\big)= n - \D$ and $A_{-\D}(G) = n$.}

     \item {$\Deg\big(A(G;x)\big) = n + \D$. Furthermore,

         \begin{equation}\label{eq:RegOrderV}
         n = \frac{\Deg_{\min}\big(A(G;x)\big) + \Deg\big(A(G;x)\big)}2
         \end{equation}
         and
         \[
         m = A_{-\D}(G) \frac{\Deg\big(A(G;x)\big) - \Deg_{\min}\big(A(G;x)\big)}4 = \frac{\Deg^2\big(A(G;x)\big) - \Deg_{\min}^2\big(A(G;x)\big)}8.
         \]
         }

     \item {$1\le A_{\D}(G) \le n/(\D+1)$. Furthermore, $G$ is a connected graph if and only if $A_{\D}(G) = 1$.}

     \item {If $\D>0$, then $A_{-\D+2}(G) \ge m$ and $A_{\D-2}(G) \ge n+n_0$ with $n_0$ the number of cut vertices; in particular, $A_{\D-2}(G)\ge n$.}

     \item {$A(G;x)$ is either an even or an odd function of $x$. Furthermore, $A(G;x)$ is an even function of $x$ if and only if $n + \D$ is even.}

     \item {The unique real zero of $A(G;x)$ is $x=0$, and its multiplicity is $n-\D$.}
   \end{enumerate}
\end{theorem}

\begin{proof}
We prove each item separately.
\begin{enumerate}[i)]
  \item {Let us consider $S\subset V$ with $S$ an exact defensive ($2i-\D$)-alliance in $G$. Then, we have for all $v\in S$
      $$2\d_S(v)\ge \d(v)+2i-\D = \D+2i-\D \quad \Leftrightarrow \quad \d_S(v)\ge i,$$
      besides, the equality holds at some $w\in S$.
      We have the result since $A_{-\D+2i}(G)$ is the number of exact defensive ($2i-\D$)-alliance in $G$.}

  \medskip

  \item {One can check directly that if $S$ is a single vertex, then $S$ is an exact defensive ($-\D$)-alliance; furthermore, it is clear that any $S\subseteq V$ with $\langle S\rangle$ connected and more than one vertex is not an exact defensive ($-\D$)-alliance, since for any $v\in S$ we have
      \begin{equation}\label{eq:i}
      \d_{S}(v) - \d_{\overline{S}}(v) \geq  1 - (\D - 1) = -\D + 2.
      \end{equation}
      Consequently $A_{-\D}(G) = n$, since $G$ is a $\D$-regular graph.}

  \medskip

  \item {The maximum value in $\mathcal{K}$ is $\D$, so $Deg\big(A(G;x)\big)$ is at most $n+\D$. We have that each connected component of $G$ is an exact defensive $\D$-alliance since $\d(v)=\D$ for any vertex $v$. Then, $A_{\D}(G)>0$ and $Deg\big(A(G;x)\big)= n+\D$. Besides, the other results are consequences of the well known fact $2m=n\D$ and the previous results.}

  \medskip

  \item{By item i), $A_{\D}(G)$ is the number of connected induced subgraphs of $G$ with minimum degree $\D$; hence, $A_{\D}(G)$ is the number of connected components of $G$. Besides, since any connected component has cardinality greater than $\D$, we obtain the upper bound of $A_{\D}(G)$.}

  \medskip

  \item {If $u,v\in V$ with $u\sim v$, then $S:=\{u,v\}$ is an exact defensive ($2-\D$)-alliance since $1=\d_S(u)=\d_{\overline{S}}(u)+2-\D$ and $1=\d_S(v)=\d_{\overline{S}}(v)+2-\D$. Thus, we obtain $A_{-\D+2}(G)\ge m$.
      Note that if $\D=1$, we have the second inequality. Assume that $\D\ge2$.
      Without loss of generality we can assume that $G$ is connected; otherwise, it suffices to analyze each connected component of $G$.
      Let us define $S_v:=V\setminus\{v\}$ for any $v\in V$.
      Since $\d_{S_v}(u)\ge \D-1$, $\d_{\overline{S_v}}(u)\le1$ for every $u\in S_v$ and both equalities hold for every $w\in N(v)$, we have that $S_v$ is an exact defensive ($\D-2$)-alliance if $v$ is a non-cut vertex, or contains at least two exact defensive ($\D-2$)-alliances if $v$ is a cut vertex.}

  \medskip

  \item {The first statement is a consequence of Theorem \ref{p:AlliPoly} iv). Consider an exact defensive $k$-alliance $S$ in $G$.
      So, there exists $v\in S$ with
      \[
      2\d_{S}(v) = \d(v) + k = \D+k.
      \]
      Then, $\D \equiv k \ (\text{mod } 2)$, $n + k \equiv n + \D \ (\text{mod } 2)$ and we have the result.}

  \medskip

  \item {Since $Deg_{\min}\big(A(G;x)\big)=n-\D$, we have that $x=0$ is a zero of $A(G;x)$ with multiplicity $n-\D$. The positivity of all coefficients of $A(G;x)$ gives $A(G;x)\neq 0$ for every $x>0$. Finally, by item vi), $A(G;x)=(-1)^{n+\D}A(G;-x)\neq0$ for every $x<0$.}
\end{enumerate}
\end{proof}

Note that Theorem \ref{t:PropRegular} (items ii, iv, v and vi) has the following direct consequence.

\begin{corollary}\label{c:Cubic}
  Let $G$ be any cubic graph. Then, $$A(G;x)=A_{-3}(G)\, x^{n-3}+A_{-1}(G)\, x^{n-1}+A_1(G)\, x^{n+1}+A_{3}(G)\, x^{n+3},$$
  with $A_{-3}(G)=n < m\le A_{-1}(G)$ and $A_{1}(G)\ge A_{3}(G)$.
\end{corollary}

A finite sequence of real numbers $(a_{0} , a_{1} , a_{2} , . . . , a_{n})$ is said to be \emph{unimodal}
if there is some $k \in \{0, 1, . . . , n\}$, called the \emph{mode} of the sequence, such that
$$a_{0} \leq . . . \leq a_{k-1} \leq a_{k} \geq a_{k+1} \geq . . . \geq a_{n}.$$
A polynomial is called unimodal
if the sequence of its coefficients is unimodal.
Therefore, we have the following result 
for alliance polynomial of a cubic graph, note that $A^\dag$ does not satisfy it.

\begin{corollary}\label{c:CubUnim}
 For any cubic graph $G$, $x^{(3-n)/2}\,A(G;\sqrt{x}\,)$ is an unimodal polynomial. 
\end{corollary}

\begin{theorem}\label{c:ConComp}
   Let $G$ be any connected graph. Then $G$ is regular if and only if $A_{\D}(G) = 1$.
\end{theorem}

\begin{proof}
   If $G$ is regular, then by Theorem \ref{p:AlliPoly} vi) we obtain $A_{\D}(G)=1$. Besides, if $A_{\D}(G)=1$, then there is an exact defensive $\D$-alliance $S$ in $G$ with $\d_S(v) \ge \d_{\bar{S}}(v) + \D \ge \D$ (i.e., $\d_S(v)=\D$ and $\d_{\overline{S}}(v)=0$) for every $v\in S$. So, the connectivity of $G$ gives that $G$ is a $\D$-regular graph.
\end{proof}

\begin{theorem}\label{t:Regulars}
   Let $G_1,G_2$ be two regular graphs. If $A(G_1;x)=A(G_2;x)$, then $G_1$ and $G_2$ have the same order, size, degree and number of connected components.
\end{theorem}

\begin{proof}
   Let $n_1,n_2$ be the orders of $G_1$ and $G_2$, respectively, and $\D_1,\D_2$ the degrees of $G_1$ and $G_2$, respectively. Then, by Theorem \ref{t:PropRegular} ii) and iii) we have
   \[
   n_1 - \D_1 = n_2 - \D_2 \quad \text{ and } \quad n_1+\D_1 = n_2+\D_2
   \]
   and we conclude
   \[
   n_1 = n_2 \quad \text{ and } \quad \D_1 = \D_2.
   \]
   Hence, both graphs have the same size.
   Finally, since $A_{\D_1}(G_1) = A_{\D_2}(G_2)$, they have the same number of connected components by Theorem \ref{p:AlliPoly} vi).
\end{proof}

\begin{corollary}\label{c:Reg1}
   Let $G_1,G_2$ be two regular graphs with orders $n_1$ and $n_2$, and degrees $\D_1$ and $\D_2$, respectively. If $n_1 \neq n_2$ or $\D_1 \neq \D_2$, then $A(G_1;x) \neq A(G_2;x)$.
\end{corollary}

%
%
%

The next theorem characterizes the degree of any regular graph by the number of non-zero coefficients of its alliance polynomial.

\begin{theorem}\label{t:PolyReg}
Let $G$ be any $\D$-regular graph with order $n$. Then $A(G;x)$ has $\D+1$ non-zero coefficients. Furthermore,
$$A(G;x)=\displaystyle\sum_{i=0}^{\D} A_{\D-2i}(G)\ x^{n+\D-2i},$$
with $A_{-\D}(G)=n$, $A_{\D}(G)\ge1$, and
$$A_{\D-2i}(G) \ge \frac{n {\D \choose i}}{\min\{\D,n-i\}} \quad \text{for  } \ 1\le i\le \D-1 \text{  if  } \D>0.$$
\end{theorem}

\begin{proof}
Since $G$ is $\D$-regular, by Theorem \ref{t:PropRegular} we have $A_{-\D}(G)=n$, $A_{\D}(G)\ge1$ and $A(G;x)$ is an even or an odd function of $x$.
Assume now that $\D>0$ and fix $1\le i \le \D-1$. Let us consider $u\in V$ and $v_1,\ldots,v_i$ different vertices in $N(u)$. Denote by $S_u:=V\setminus\{v_1,\ldots,v_i\}$.
Then, we have that $\d_{S_u}(v)\ge \D-i$ and $\d_{\overline{S_u}}(v)\le i$ for every $v\in S_u$; furthermore, the equalities hold at $u$.
Let $S_u^*\subset S_u$ such that $\langle S_u^*\rangle$ is the connected component of $\langle S_u\rangle$ which contains $u$. So, $S_u^*$ is an exact defensive ($\D-2i$)-alliance and $A_{\D-2i}(G)>0$.
Since each set $S_u^*$ can appear at most $n-i$ times (once for each $S_w^*$ with $w\in V\setminus \{v_1,\ldots,v_i\}$), and at most $\D$ times (once for each $S_w^*$ with $w\sim v_1$), we obtain $A_{\D-2i}(G)\ge n {\D \choose i}/ \min\{\D,n-i\}$.
\end{proof}

A \emph{Hamiltonian cycle} is a cycle in a graph that visits each vertex exactly once.
A graph that contains a Hamiltonian cycle is called a \emph{Hamiltonian graph}.
The following theorem is a well known result in graph theory which will be useful.

\begin{theorem}[Dirac 1952]\label{Dirac}
A graph with order $n \ge 3$ is Hamiltonian if every vertex has degree $n / 2$ or greater.
\end{theorem}

In what follows we will use the following notation: for any $A,B\subset V$, we denote by $N(A,B)$ the number of edges with one endpoint in $A$ and the other endpoint in $B$.

\begin{theorem}\label{t:An+D-2=n}
  Let $G$ be any $\D$-regular graph with order $n < 2\D$. Then $A_{\D-2}(G)=n$.
\end{theorem}

\begin{proof}
Notice that $\D\ge 2$, since otherwise, such a graph $G$ does not exist; furthermore, $n\ge \D+1 \ge 3$.
We have that $G$ is a Hamiltonian graph by Theorem \ref{Dirac}.
Besides, by Theorem \ref{t:PropRegular} i), we have that $A_{\D-2}(G)$ is the number of connected induced subgraphs of $G$ with minimum degree $\D-1$. Let us consider $u\in V$ and define $S_u:=V\setminus\{u\}$.
Since $G$ is a Hamiltonian graph, $\langle S_u\rangle$ is connected. Besides, we have $\d_{S_u}(v)\ge \D-1 \ge \d_{\overline{S_u}}(v) + \D-2$ for all $v\in S_u$ and the equality holds at $w\in N(u)$. So, $S_u$ is an exact defensive ($\D-2$)-alliance in $G$ and $A_{\D-2}(G)\ge n$.

Seeking for a contradiction assume that there is an exact defensive ($\D-2$)-alliance $S\subset V$ with $|S|\le n-2$. Notice that $|S|\ge\D>n/2$, by Theorem \ref{t:PropRegular} i).
Then, since any vertex in $S$ has degree $\D$ in $G$ with at most one edge among $S$ and $\overline{S}$, we have
$$N(S,S) + N(S,\overline{S}) = \frac{|S| \D}2 + \frac{N(S,\overline{S})}2 \le \frac{|S| \D}2 + \frac{|S|}2 = \frac{|S|(\D+1)}2.$$
Besides, since $|\overline{S}|=n-|S|$, we have
\[
N(\overline{S},\overline{S}) \le \frac{(n-|S|)(n-|S|-1)}2.
\]
If $m$ denotes the size of $G$, then
\[
\begin{aligned}
0 & = 2\left(N(S,S) + N(S,\overline{S}) + N(\overline{S},\overline{S})\right) - 2m \\
&\le |S|(\D+1) + (n-|S|)(n-|S|-1) - n\D \\
& = |S|^2 + |S|(\D+2-2n) + n^2 - n - n\D.
\end{aligned}
\]
Define $P(x):=x^2 + x(\D+2-2n) + n^2 - n - n\D$; then $P(|S|)\ge0$. Since
\[
\begin{aligned}
P\left(\frac{n}2\right) &= \frac{n^2}4 + \frac{n}2 (\D+2-2n) + n^2 - n - n\D \\
&= \frac{n^2}4 + \frac{n\D}2 + n -n^2 + n^2 - n - n\D \\
&= \frac{n}4(n-2\D)<0
\end{aligned}
\]
and
\[
\begin{aligned}
P(n-2)&= (n-2)^2 + (n-2)(\D+2-2n) + n^2 - n - n\D \\
&= (n-2)^2 + (n-2)(\D-n) - (n-2)^2 + n^2 - n - n\D \\
&= n-2\D < 0,
\end{aligned}
\]
we obtain that $P(|S|)<0$. This is the contradiction we were looking for, so, there not exists an exact defensive ($\D-2$)-alliance $S$ with $|S|\le n-2$.
This finishes the proof since $V$ is an exact defensive $\D$-alliance.
\end{proof}

A \emph{clique} in a graph $G = (V, E)$ is a subset $C$ of the vertex set $V$, such that $\langle C\rangle$ is a complete graph.

\begin{lemma}\label{l:clique}
 Let $G$ be any $\D$-regular graph with order $n$, $\D\ge3$ and $2\D\le n \le 2\D+1$. If $G$ contains two cliques of cardinality $\D$, then these cliques are disjoint.
 In particular, $G$ contains at most two cliques of cardinality $\D$.
\end{lemma}

\begin{proof}
Seeking for a contradiction, assume that there exist $S_1,S_2\subset V$ cliques of cardinality $\D$ with $S_1\cap S_2\neq\emptyset$. Denote by $r$ the number $r:= |S_1\cap S_2|$; then $1\le r \le \D-1$. Note that for any $v\in S_1\cap S_2$ we have $\d_{S_1\cup S_2}(v)= |S_1|-1 + |S_2\setminus S_1|= \D-1+\D-r$, so, we obtain $r=\D-1$.
Then, we have $|S_1\cup S_2|=\D+1$ and $\D-1\le|\overline{S_1\cup S_2}|\le \D$. Besides, we have $N(S_1\cup S_2,\overline{S_1\cup S_2}) = 2 = |(S_1\cup S_2)\setminus (S_1\cap S_2)|$ and, since $|\overline{S_1\cup S_2}|\le\D$, $N(\overline{S_1\cup S_2},S_1\cup S_2)\ge |\overline{S_1\cup S_2}|\ge \D-1$.
Since $N(S_1\cup S_2,\overline{S_1\cup S_2}) = N(\overline{S_1\cup S_2},S_1\cup S_2)$, we obtain $\D=3$ and $n=6$; therefore, $G$ is a graph isomorphic to either $K_{3,3}$ or the Cartesian product $P_2\Box K_3$. Thus, we obtain that there are not two non-disjoint cliques in $G$ with cardinality $\D$.
This finishes the proof since, by $n\le 2\D+1$, it is impossible to have three disjoint cliques of cardinality $\D$ contained in $G$.
\end{proof}

\begin{remark}\label{r:exitClique}
If $G$ is a $\D$-regular graph with $n \le 2\D+1$, then
$G$ does not contain a clique of cardinality greater than $\D$, since $2(\D+1)>2\D+1\ge n$.
\end{remark}

\begin{remark}\label{r:DisjointClique}
Let $G$ be any $\D$-regular graph with order $n$ and $\D\ge1$ such that $G$ has two disjoint cliques of cardinality $\D$. Then
\begin{enumerate}
  \item If $n=2\D$, then $G$ is isomorphic to the Cartesian product graph $P_2 \Box K_\D$.
  \item If $n=2\D+1$, then $\D$ is even (since $n\D=2m$) and $G$ can be obtained from $P_2 \Box K_\D$ by removing $\D/2$ copy edges of $P_2$ and connecting the $\D$ vertices with degree $\D-1$ with a new vertex. In particular, if $S$ is a clique of cardinality $\D$ in $G$, then $\overline{S}$ is not an exact defensive ($\D-2$)-alliance.
\end{enumerate}
\end{remark}

\begin{theorem}\label{t:An+D-2<=n+m}
  Let $G$ be any $\D$-regular graph with order $n$, size $m$, $\D\ge 3$ and $2\D\le n \le 2\D+1$. Then $n\le A_{\D-2}(G)\le n+m+2$.
\end{theorem}

\begin{proof}
Note that if $\D=3$ then $n=6$, and $G$ is a graph isomorphic to either $K_{3,3}$ or $P_2\Box K_3$. Thus, a simple computation gives $6\le A_{1}(K_{3,3})=15\le 6 + 9 + 2$ and $6\le A_{1}(P_2\Box K_3) = 11 \le 6 + 9 + 2$.

Assume now that $\D\ge 4$.
Clearly, $G$ is a connected graph and $\diam G = 2$, since $2\D>n-2$.

First we prove that $G$ does not have cut vertices. If $n=2\D$, then $G$ is a Hamiltonian graph by Theorem \ref{Dirac}.
If $n=2\D+1$, seeking for a contradiction assume that there is a cut vertex $w$ in $G$. Let $S_1,S_2 \subset V$ with $S_1\cup S_2\cup \{w\}=V$ such that $\langle S_1\rangle$ and $\langle S_2\rangle$ are disjoint.
Without loss of generality we can assume that $|S_1|\le \D \le |S_2|$. Since $\d_{S_1}(w),\d_{S_2}(w)\ge1$, $\d_{S_1}(w)+\d_{S_2}(w)=\D$ and $\d_{S_1}(u)\le |S_1|-1 \le \D-1$ for all $u\in S_1$, we have $\d_{S_1}(w)=|S_1|$ and $\d_{S_1}(u)=\D-1$ for all $u\in S_1$. Then, we obtain that $|S_1|=\D$, but this is a contradiction since $\d_{S_1}(w)=\D-\d_{S_2}(w)\le \D-1<\D=|S_1|=\d_{S_1}(w)$.
Then, $G$ does not have cut vertices.

By Theorem \ref{t:PropRegular} i), we have that $A_{\D-2}(G)$ is the number of connected induced subgraphs of $G$ with minimum degree $\D-1$; thus, any exact defensive ($\D-2$)-alliance $S$ in $G$ verifies $|S|\ge \D$.
Let us consider $u\in V$ and denote by $S_u:=V\setminus\{u\}$.
Since $G$ does not have cut vertices, $\langle S_u\rangle$ is connected. Besides, we have $\d_{S_u}(v)\ge \D-1 \ge \d_{\overline{S_u}}(v) + \D-2$ for all $v\in S_u$ and the equality holds for every $v\in N(u)$; so, $S_u$ is an exact defensive ($\D-2$)-alliance in $G$. Thus, $A_{\D-2}(G)\ge n$.

Let us consider $u_1,u_2\in V$ with $u_1\neq u_2$ and define $S_{u_1,u_2}:=V\setminus\{u_1,u_2\}$. If $u_1\nsim u_2$, then there is $w\in V$ with $u_1,u_2\in N(w)$ since $\d(u_1)+\d(u_2)= 2\D >|S_{u_1,u_2}|$; in fact, $S_{u_1,u_2}$ is not a defensive ($\D-2$)-alliance in $G$.
So, $S_{u_1,u_2}$ may be an exact defensive ($\D-2$)-alliance in $G$, if $u_1\sim u_2$; then there are at most $m$ exact defensive ($\D-2$)-alliances with $n-2$ vertices.
Consider now $u_1,\ldots,u_r\in V$ with $3\le r\le \D-1$ and $u_i\neq u_j$ if $i\neq j$.
Note that $S_r:=V\setminus\{u_1,\ldots,u_r\}$ is not a defensive ($\D-2$)-alliance in $G$ if $r>3$, since $N(\overline{S_r},S_r)\ge r(\D-r+1)=2\D-r+(r-2)(\D-r)> 2\D+1-r \ge |S_r|$. Besides, if $r=3$ and $\D\ge5$ (thus $\D-r\ge2$) we have the same inequality and then $S_r$ is not a defensive ($\D-2$)-alliance in $G$.
Note that, if $r=3$ and $n=2\D$, then $N(\overline{S_r},S_r)\ge 2\D-r+(r-2)(\D-r)>2\D-r=n-r\ge |S_r|$ and we also conclude that $S_r$ is not a defensive ($\D-2$)-alliance in $G$.
However, if $r=3$, $\D=4$ and $n=2\D+1$ (thus, $n=9$), then $S_r$ may be an exact defensive ($\D-2$)-alliance in $G$. But a simple computation gives that these five graphs $G$ verify $A_{2}(G)< 9+18+2$.

We analyze separately the cases $n=2\D$ and $n=2\D+1$.
Assume first that $n=2\D$. We only need to compute the possible exact defensive ($\D-2$)-alliances in $G$ with cardinality $\D$, since every defensive ($\D-2$)-alliance has at least $\D$ vertices and $n=2\D$.
If $S$ is an exact defensive ($\D-2$)-alliance in $G$, then $S$ is a clique of cardinality $\D$ and by Lemma \ref{l:clique} there are at most $2$ exact defensive ($\D-2$)-alliances with $\D$ vertices.
Assume now that $n=2\D+1$. So, $\D$ is even.
We only need to compute the possible exact defensive ($\D-2$)-alliances in $G$ with cardinalities $\D$ and $\D+1$.
If $S$ is an exact defensive ($\D-2$)-alliance in $G$ with $|S|=\D+1$, then $\d_S(u)\ge \D-1$ for every $u\in S$ and $\d_S(u_0)=\D$ for some $u_0\in S$, since otherwise $\d_S(u)=\D-1$ for every $u\in S$ and we conclude $(\D+1)(\D-1)=|S|(\D-1)=2m_S$, with $m_S$ the size of $\langle S\rangle$, which is not possible since $\D$ is even. Hence, $N(\overline{S},S)\le \D$; furthermore, since $|\overline{S}|=\D$, $\d_{S}(v)\ge 1$ for all $v\in \overline{S}$, and so, $\overline{S}$ is a clique.
If $S$ is an exact defensive ($\D-2$)-alliance in $G$ with $|S|=\D$, then $\d_S(u)\ge\D-1$ for every $u\in S$ and $S$ is a clique of cardinality $\D$. Lemma \ref{l:clique} completes the proof since if $G$ has two cliques of cardinality $\D$, then they are disjoint and Remark \ref{r:DisjointClique} gives that $\overline{S}$ is not an exact defensive ($\D-2$)-alliance in $G$.
\end{proof}

\begin{theorem}\label{t:RegCharact}
   Let $G$ be a $\D$-regular connected graph with order $n$ and let $G^*$ be a graph with order $n_1$ and, minimum and maximum degrees $\d_1$ and $\D_1$, respectively. If $A(G^*;x)=A(G;x)$, then $G^*$ is a connected graph with exactly $n$ vertices of degree $\D_1=\D+n_1-n$, $n_1\ge n$, $\D_1\ge\D$ and $\d_1 \equiv \D_1 (\text{mod } 2)$.

   Furthermore, if $n_1 > n$, then the following inequalities hold:
     \begin{equation}\label{eq:Poly2}
       \frac{\D_1+\d_1+2}2\le \D.
     \end{equation}
     \begin{equation}\label{eq:Poly1}
       \d_1+2 < \D < \D_1,
     \end{equation}
     \begin{equation}\label{eq:Poly3}
       \D + 1 \le \D_1 \le 2\D-3,
     \end{equation}
     \begin{equation}\label{eq:Poly4}
       \d_1+4 \le \D_1.
     \end{equation}

\end{theorem}


\begin{proof}
   Since $A(G^*;x)=A(G;x)$ is a symmetric polynomial by Theorem \ref{t:PropRegular} vi), we conclude that $\d_1 \equiv \D_1 (\text{mod } 2)$ by Theorem \ref{p:AlliPoly} iv).
   By Theorems \ref{p:AlliPoly} v) and \ref{t:PropRegular} ii), $G^*$ has $n$ vertices of maximum degree $\D_1$, so, $n_1 \geq n$; besides, $n_1 - \D_1 = n-\D$. Note that if $n_1=n$ then $G^*$ is a $\D$-regular graph with $A_{\D}(G^*)=1$, so, Theorem \ref{c:ConComp} gives that $G^*$ is a connected graph.

   Assume that $n_1 > n$. Denote by $t:=n_1-n=\D_1 -\D$.
   Let $v_1,\ldots,v_n \in V(G^*)$ be the vertices in $G^*$ with degree $\D_1$ and define $S:=\{v_1,\ldots,v_n\}$.
   Note that for any $v\in S$ we have $\d_{S}(v) \ge \D_1 - t = t + (\D_1-2t) \ge \d_{\overline{S}}(v) + \D_1-2t$; hence, $S$ contains a defensive ($\D_1-2t$)-alliance $S_1$ and $k_{S_1} \ge \D_1-2t$. Therefore, there is at least one term of degree greater or equal than $n_1+\D_1-2t$ in $A(G^*;x)$.
   Since $x^{n_1+\D_1-2t}=x^{n+\D}$, $S_1$ is an exact defensive ($\D_1-2t$)-alliance in $G^*$.
   Finally, note that if $\langle S\rangle$ is not a connected subgraph (i.e., $S_1\neq S$), then in $A(G^*;x)$ appear at least two terms $x^{n+\D}$, but this is a contradiction since $A(G;x)$ is a monic polynomial by Theorem \ref{p:AlliPoly} vi).
   Hence, $\langle S\rangle$ is connected. Since the degree of $A(G^*;x)$ is $n+\D=n_1+\D_1-2t$, then $S$ is an exact defensive ($\D_1-2t$)-alliance in $G^*$; therefore, there exists $1\le j\le n$ such that $\D_1=\d(v_j)=2\d_{\overline{S}}(v_j)+\D_1-2t$, and we have $\d_{\overline{S}}(v_j)=t$.
   Since $|S|=n=n_1-t$ and $|\overline{S}|=t$, $\overline{S}\subseteq N(v_j)$ and $G^*$ is a connected graph.

   Also, since $G^*$ is connected, $A(G^*;x)=A(G;x)$, $k_S=\D_1-2t$ and $k_{V(G^*)}=\d_1$, we have $\d_1\le \D_1-2t$. We are going to prove $\d_1 < \D_1 - 2t$; seeking for a contradiction assume that $\d_1 = \D_1 - 2t$. Since $G^*$ is connected, $k_{V(G^*)}=\d_1= \D_1 - 2t = k_{S}$ and this contradicts that $A(G^*;x)$ is a monic polynomial. Therefore, $\d_1 < \D_1 - 2t$. But, since $\d_1\equiv \D_1 (\text{mod } 2)$ we obtain $\d_1+2\le \D_1-2(\D_1-\D)=2\D-\D_1$, so \eqref{eq:Poly2} holds.

   Besides, since $\D_1>\D$, \eqref{eq:Poly2} gives $\d_1+2 <\D$, and so, \eqref{eq:Poly1} holds.
   Furthermore, we have $\D+1\le \D_1$ and \eqref{eq:Poly2} gives \eqref{eq:Poly3}, since $\d_1\ge1$.
   Finally, since $\D\le \D_1-1$, \eqref{eq:Poly2} provides \eqref{eq:Poly4}.
\end{proof}

\section{Alliance polynomials of regular graphs with small degree}\label{sect:con_reg}
The theorems in this section can be seen as a natural continuation of the study in \cite{CRST} in the sense of showing the distinctive power of the alliance polynomial of a graph.
In particular, we show that the family of alliance polynomials of $\D$-regular graphs with small degree $\D$ is a special family of alliance polynomials since there not exists a non $\D$-regular graph with alliance polynomial equal to one of their members, see Theorems \ref{t:Reg0-3} and \ref{t:Reg4-6}.

\begin{theorem}\label{t:Reg0-3}
  Let $G$ be a $\D$-regular graph with $0\le\D\le3$ and $G^*$ another graph. If $A(G^*;x) = A(G;x)$, then $G^*$ is a $\D$-regular graph with the same order, size and number of connected components of $G$. 
\end{theorem}

\begin{proof}
In \cite[Theorem 2.7]{CRST} the authors obtain the uniqueness of the alliance polynomials of $0$-regular graphs (the empty graphs).

\medskip

Theorems \ref{p:AlliPoly} iii) and \ref{t:PolyReg} give that $1$-regular graphs are the unique graphs which have exactly two non-zero terms in their alliance polynomial; besides, Theorems \ref{p:AlliPoly} vi) and \ref{t:PropRegular} ii) give the uniqueness of these alliance polynomials.

\medskip

In order to obtain the result for $2\le\D\le3$, denote by $n,n_1$, the orders of $G,G^*$, respectively, and let $\d_1,\D_1$ be the minimum and maximum degree of $G^*$.

\smallskip

Assume first that $\D=2$.
By Theorem \ref{t:PolyReg} we have $A(G;x)= n x^{n-2} + A_0(G) x^n + A_2(G) x^{n+2}$, thus, by Theorem \ref{p:AlliPoly} iii) the degree sequence of $G^*$ has at most two different values.
If $G^*$ is regular then Theorem \ref{t:Regulars} gives the result. Therefore, seeking for a contradiction assume that the degree sequence of $G^*$ has exactly two different values (i.e., $G^*$ is bi-regular).
By Theorems \ref{p:AlliPoly} iv) and \ref{t:PropRegular} vi) we have $\d_1 \equiv \D_1 (\text{mod } 2)$.
By Theorems \ref{p:AlliPoly} v) and \ref{t:PropRegular} ii) we have $A_{-\D_1}(G^*)=A_2(G)=n<n_1$ and $n-2=n_1-\D_1$, so, we have $\D_1>2$.
By Theorems \ref{p:AlliPoly} vii) and \ref{t:PropRegular} iii) we have $n_1+\d_1 \le n+2$, so, we obtain $0\le \d_1 \le 1$.
If $\d_1=0$, then there is a connected component $G'$ of $G^*$ which is $\D_1$-regular. So, $k_{V(G')}=\D_1$ and $\Deg\big(A(G^*;x)\big)=n_1+\D_1 > n+2$, which is a contradiction. Thus, we can assume that $\d_1=1$.
Then, we have $n_1=n+1$; and so, $\D_1=3$.
We prove now that $A_1(G^*)\ge n$.
Let $u_0,v_0$ be the vertices of $G^*$ with $\d(u_0)=1$ and $v_0\sim u_0$.
If $G^*$ is not connected, then it has a $3$-regular connected component $G_0^*$; since $V(G_0^*)$ is an exact defensive $3$-alliance, then $\Deg \big( A(G^*;x)\big) \ge n_1+3 > n+2 = \Deg \big( A(G;x)\big)$, which is a contradiction and we conclude that $G^*$ is connected.
Let us define $S_v:=V(G^*)\setminus\{v\}$ for any $v\in V(G^*)\setminus \{v_0\}$. Since $\d_{S_v}(u)\ge 2$, $\d_{\overline{S_v}}(u)\le1$ for every $u\in S_v\setminus\{u_0\}$ and both equalities hold for every $w\in N(v)$, and $\d_{S_v}(u_0)=1$, $\d_{\overline{S_v}}(u_0)=0$, we have that $S_v$ is an exact defensive $1$-alliance or contains an exact defensive $1$-alliance if $v$ is a cut vertex. Thus, $A_{1}(G^*)\ge n$.
Besides, Theorem \ref{t:PropRegular} iv) gives $A_2(G)\le n/3 < n\le A_{1}(G^*)$, so, $A(G;x)\neq A(G^*;x)$. This is the contradiction we were looking for, and so, we conclude $n_1=n$ and $\D_1=2$, and we obtain the result for $\D=2$.

\medskip

Finally, assume that $\D=3$. By Corollary \ref{c:Cubic} we have $A(G^*;x) = A(G;x)= n x^{n-3} + A_{-1}(G) x^{n-1} + A_{1}(G) x^{n+1} + A_{3}(G) x^{n+3}$, with $A_3(G)$ is the number of connected components of $G$.
By Theorem \ref{p:AlliPoly} i), we have $n_1 - \D_1 = n - 3$ and $n \leq n_1$. Hence, $n_1\geq n$ and $\D_1 \geq 3$.
Also we have $n_1 + \d_1 \le n+3$ by Theorem \ref{p:AlliPoly} vii). Furthermore, if $\D_1 = 3$, then $n_1 = n$ and so, $G^*$ is $3$-regular since $A_{- 3}(G^*) = n$. By Theorem \ref{p:AlliPoly} vi) they have the same number of connected components, and consequently $G,G^*$ have the same size, too. We will finish the proof by checking that $\D_1=3$.

Seeking a contradiction, assume that $\D_1 > 3$ (then $n_1 > n$) and let $k = n_1 - n = \D_1 - 3$.

\smallskip

Assume that $\D_1 \ge 6$ (i.e., $k \ge 3$).
Then there exists a connected component $G_0$ of $G^*$ with $\d_{G_0}(v)=\d(v)\ge1$ for every $v\in V(G_0)$; if $S=V(G_0)$, then $\d_S(v)=\d(v)\ge1$, and so, $k_{S}^{(G^*)} \ge 1$.
Hence, $A(G^*;x)$ has at least one term with exponent greater than $n_1 \ge n + 3 = \Deg\big(A(G;x)\big)$, and $A(G^*;x) \neq A(G;x)$, which is a contradiction.
Thus, $\D_1=4$ or $\D_1=5$.

\smallskip

Assume that $\D_1 = 5$, then $n_1 = n + 2$. By Theorem \ref{p:AlliPoly} v), we have that $G^*$ has exactly $n$ vertices with degree $5$; and so, by Theorem \ref{p:AlliPoly} iv), we have that the other two vertices of $G$ have degree $1$ or $3$. Since $n_1+\d_1\le n+3$ by Theorem \ref{p:AlliPoly} vii), we obtain $\d_1=1$.

Assume that $G^*$ has two vertices $v_1$ and $v_2$ with degree $1$. In this case, if $v_1 \sim v_2$, then $G^*$ is a disconnected graph with at least one connected component which is $5$-regular since $V(G^*)\setminus \{v_1,v_2\}$ induces a $5$-regular subgraph $G_1$ of $G^*$. Since $V(G_1)$ is an exact defensive $5$-alliance, $\Deg(A(G^*;x))\ge n_1+5$ and we have $\Deg(A(G^*;x)) \ge n_1 + 5 > n+3 = \Deg(A(G;x))$.
If $v_1 \nsim v_2$ but there exists $w\in V(G^*)$ such that $w\sim v_1$ and $w\sim v_2$, then let us consider the connected component $G_2$ of $G^*$ containing $\{v_1,v_2,w\}$. The set $S=V(G_2)\setminus\{v_1,v_2,w\}$ is a defensive $3$-alliance in $G^*$, since for any $v\in S$ we have $\d_{S}(v) \geq 4$ and $\d_{\overline{S}}(v) \leq 1$. Then, $\Deg(A(G^*;x)) \geq n_1 + 3 > n + 3 = \Deg(A(G;x))$.
If $v_1 \nsim v_2$ but there does not exist $w\in V(G^*)$ with $w\sim v_1$ and $w\sim v_2$, then let us consider the connected component $G_3$ of $G^*$ containing $v_1$ and $S=V(G_3)\setminus\{v_1,v_2\}$. The set $S$ is a defensive $3$-alliance in $G^*$, since for all $v\in S$ we have $\d_{S}(v) \geq 4$ and $\d_{\overline{S}}(v) \leq 1$. Then, $\Deg(A(G^*;x)) \geq n_1 + 3 > n + 3 = \Deg(A(G;x))$.

Consider now the case of $G^*$ containing two vertices $v_1$ and $v_2$ with degree $1$ and $3$, respectively.
If $v_1 \sim v_2$, then let us consider the connected component $G_4$ of $G^*$ containing $\{v_1,v_2\}$ and $S=V(G_4)\setminus\{v_1,v_2\}$.
Then, $S$ is a defensive $3$-alliance in $G^*$, since for all $v\in S$ we have $\d_{S}(v) \geq 4$ and $\d_{\overline{S}}(v) \leq 1$. Then, $\Deg(A(G^*;x)) \geq n_1 + 3 > n + 3 = \Deg(A(G;x))$.
If $v_1 \nsim v_2$, let $G_5$ be the connected component of $G^*$ containing $v_1$ and $S=V(G_5)\setminus\{v_1\}$. Hence, $S$ is an exact defensive $3$-alliance in $G^*$, since $\d_{S}(v_2) - \d_{\overline{S}}(v_2) = 3-0$ if $v_2\in S$ and $\d_{S}(v) - \d_{\overline{S}}(v) \ge 4 - 1$ for any $v\in S\setminus\{v_2\}$.
Then, $\Deg(A(G^*;x)) \geq n_1 + 3 > n + 3 = \Deg(A(G;x))$.
So, it is not possible to have $\D_1 = 5$.

\smallskip

Assume that $\D_1 = 4$, then $n_1 = n + 1$.
If $G^*$ is a disconnected graph, then there exists a connected component $\langle S^*\rangle$ of $G^*$ such that $\langle S^*\rangle$ is $4$-regular and so, $S^*$ is an exact defensive $4$-alliance in $G^*$. Therefore, $\Deg(A(G^*;x)) = n_1 + 4 > n + 3 = \Deg(A(G;x))$. Thus, $G^*$ is connected, and $\d_1 = 2$ by Theorem \ref{p:AlliPoly} iv).
So, we have that $G^*$ has exactly $n$ vertices with degree $4$ and another vertex $w$ with degree $2$.
Let $v_1,v_2 \in V(G^*)\setminus\{w\}$ with $v_1\neq v_2$, $v_1\sim w$ and $v_2\sim w$. Consider $\{u_1,\ldots,u_{n-2}\} := V(G^*)\setminus \{w,v_1,v_2\}$.
Let $G_i$ be the connected component of $\langle V(G^*)\setminus\{u_i\}\rangle \subset G^*$ containing $w$, and $S_i=V(G_i)$, for each $1\le i\le n-2$.
Note that $S_i$ is an exact defensive $2$-alliance since $\d_{S_i}(w)-\d_{\overline{S_i}}(w)=2$, for each $1\le i\le n-2$.
Note that if $i\neq j$ and $u_j\notin S_i$ then $u_i \in S_j$, and so, $S_i\neq S_j$ since $u_i\notin S_i$; furthermore, if $u_j\in S_i$ then $S_i\neq S_j$ since $u_j\notin S_j$.
Then, we obtain that $A_{2}(G^*)\ge n-1$, and thus $A_{3}(G)\ge n-1$. This contradicts Theorem \ref{p:AlliPoly} vi) since $G$ is a cubic graph with order $n$.
So, it is not possible to have $\D_1 = 4$.
\end{proof}

Now we prove a similar result for $\D$-regular graphs with $\D>3$. First, we prove some technical results which will be useful.

\begin{lemma}\label{l:d1=1}
  Let $G_1$ be a graph with minimum and maximum degree $\d_1$ and $\D_1$, respectively, and let $n\ge3$ be a fixed natural number. Assume that $G_1$ has order $n_1 > n$ with exactly $n$ vertices of degree $\D_1$, and such that its alliance polynomial $A(G_1;x)$ is symmetric. The following statements hold:
  \begin{enumerate}
    \item {If $\d_1=1$, then $A(G_1;x)$ is not a monic polynomial of degree $2n-n_1+\D_1$.}
    \item {If $\d_1=2$, then we have $2n_1<2\D_1+n$ or $A(G_1;x)$ is not a monic polynomial of degree $2n-n_1+\D_1$.}
  \end{enumerate}
\end{lemma}

\begin{proof}
Seeking for a contradiction assume that $A(G_1;x)$ is a monic polynomial with degree $2n-n_1+\D_1$.
By hypothesis, we have $n$ different vertices $v_1,\ldots,v_n$ in $G_1$ with degree $\D_1$. Denote by $S$ the set $S:=\{v_1,\ldots,v_n\}$.
The argument in the proof of Theorem \ref{t:RegCharact} gives that $G_1$ is a connected graph, $S$ is an exact defensive \big[$\D_1-2(n_1-n)$\big]-alliance in $G_1$ and there is $w\in S$ with $\overline{S}\subseteq N(w)$.
Let $u\in \overline{S}$ with $\d(u)=\d_1$.

First assume that $\d_1=1$.
So, $S_w:=S\setminus\{w\}$ contains a defensive \big[$\D_1-2(n_1-n)$\big]-alliance since $\d_{S_w}(v)\ge \D_1 - \big(|\overline{S}\cup\{w\}|-|\{u\}|\big) = \D_1 - (n_1-n)$ and $\d_{\overline{S}_w}(v)\le |\overline{S}\cup\{w\}|-1 = n_1-n$ for all $v\in S_w$; thus, in $A(G_1;x)$ appears at least one term of degree greater or equal than $2n-n_1+\D_1$ associated to $S_w$, but this is impossible since $A(G_1;x)$ is monic of degree $2n-n_1+\D_1$. This is the contradiction we were looking for.

Assume now that $\d_1=2$. Let $w'\in V(G_1)\setminus\{w\}$ with $w'\sim u$.
If $w'\in \overline{S}$ then $S_w$ is a defensive \big[$\D_1-2(n_1-n)$\big]-alliance since $u\notin N(v)$ for every $v\in S_w$.
This implies a contradiction as above.
So, we can assume that $w'\in S_w$.
Note that if $w'\nsim w$ then $S_w$ is a defensive \big[$\D_1-2(n_1-n)$\big]-alliance since $\d_{S_w}(w')-\d_{\overline{S_w}}(w') \ge (\D_1-n_1+n)-(n_1-n)$ and $\d_{S_w}(v)-\d_{\overline{S_w}}(v) \ge (\D_1-n_1+n)-(n_1-n)$ for all $v\in S_w\setminus\{w'\}$, but this is impossible since $A(G_1;x)$ is a monic polynomial of degree $n_1+\D_1-2(n_1-n)$.
Then, we can assume that $w'\sim w$. Note that if $\d_{\overline{S}}(w')< n_1-n$ then $S_w$ is a defensive \big[$\D_1-2(n_1-n)$\big]-alliance, but this is impossible, too.
So, we can assume that $\overline{S} \subseteq N(w')$.
Notice that if there is $u'\in S$ with $d(u',\{w,w'\})\ge 2$, then we can check that $S\setminus\{u'\}$ is a defensive \big[$\D_1-2(n_1-n)$\big]-alliance, which is impossible.
Thus, we can assume that $S\subseteq N(w)\cup N(w')$; in fact,
\[
n-2=|S\setminus\{w,w'\}| \le \d_{S\setminus\{w'\}}(w) + \d_{S\setminus\{w\}}(w') = 2[\D_1 -(n_1-n)-1].
\]
Since $S\subseteq N(w)\cup N(w')$, if $n-2=2[\D_1 -(n_1-n)-1]$ then $S\cap N(w)\cap N(w')=\emptyset$, and
\[
\d_{S\setminus\{w,w'\}}(v) \ge \D_1-(n_1-n) \text{  and  } \d_{\overline{S\setminus\{w,w'\}}}(v) \le n_1-n,\quad \text{for every } v\in S\setminus\{w,w'\}.
\]
Hence, $S\setminus\{w,w'\}$ is a defensive \big[$\D_1-2(n_1-n)$\big]-alliance, which is impossible. Then $n-2 < 2[\D_1 -(n_1-n)-1]$ and this finishes the proof.
\end{proof}

\begin{lemma}\label{l:An+D-2>n}
  Let $G_1$ be a graph with minimum and maximum degree $2$ and $\D_1$, respectively, and let $n\ge3$ be a fixed natural number. Assume that $G_1$ has order $n_1 > n$ with exactly $n$ vertices of degree $\D_1$, and such that its alliance polynomial $A(G_1;x)$ is symmetric. If $n< 2[\D_1 -(n_1-n)]$ and $A(G_1;x)$ is a monic polynomial of degree $2n-n_1+\D_1$, then $A_{2(n-n_1)+\D_1-2}(G_1)> n$.
\end{lemma}

\begin{proof}
By hypothesis, there exist different vertices $v_1,\ldots,v_n$ in $G_1$ with degree $\D_1$. The arguments in the proof of Lemma \ref{l:d1=1} give that $G_1$ is a connected graph where $S:=\{v_1,\ldots,v_n\}$ is the unique exact defensive \big[$\D_1-2(n_1-n)$\big]-alliance in $G_1$ and there are $w,w'\in S$ with $\overline{S}\subset N(w)\cap N(w')$.
Note that $S_u:=S\setminus\{u\}$ is a defensive \big[$\D_1-2(n_1-n)-2$\big]-alliance for any $u\in S$, since for all $v\in S_u$ we have
\[
\d_{S_u}(v)\ge \D_1 - \big|\overline{S_u}\big| \quad \text{ and } \quad \d_{\overline{S_u}}(v)\le \big|\overline{S_u}\big|=n_1-n+1.
\]
Note that $\d_{S}(v)\ge \D_1-(n_1-n) > n/2$ for every $v\in S$.
Since $\langle S\rangle$ is Hamiltonian by Theorem \ref{Dirac}, we have that $S_u$ induces a connected subgraph for any $u\in S$.
Since $S$ is the unique exact defensive \big[$\D_1-2(n_1-n)$\big]-alliance in $G_1$, $S_u$ is an exact defensive \big[$\D_1-2(n_1-n)-2$\big]-alliance for any $u\in S$.
Therefore, we have $A_{\D_1-2(n_1-n)-2}(G_1)\ge n$.

Denote by $u'$ a vertex of $G_1$ with $\d(u')=2$. Since $v\nsim u'$ for any $v\in S\setminus\{w,w'\}$ we have $|S|-1\ge \d_{S}(v)\ge\d_{S}(w)+1$, and so, $\d_S(w)\le |S|-2$ and there are $u_1,u_2\in S\setminus\{w,w'\}$ with $u_1,u_2\notin N(w)$; then $u_1,u_2\notin N(w)\cap N(w')$. Note that $S\setminus\{u_1,u_2\}$ is a defensive \big[$\D_1-2(n_1-n)-2$\big]-alliance in $G_1$, since
$$\d_{S\setminus\{u_1,u_2\}}(w) - \d_{\overline{S\setminus\{u_1,u_2\}}}(w) = \D_1 - 2\d_{\overline{S\setminus\{u_1,u_2\}}}(w) \ge \D_1 - 2(n_1-n+1),$$
$$\d_{S\setminus\{u_1,u_2\}}(w') - \d_{\overline{S\setminus\{u_1,u_2\}}}(w')  = \D_1 - 2\d_{\overline{S\setminus\{u_1,u_2\}}}(w') \ge \D_1 - 2(n_1-n+1),$$ and
$$\d_{S\setminus\{u_1,u_2\}}(v) - \d_{\overline{S\setminus\{u_1,u_2\}}}(v) \ge \D_1 - 2( n_1-n+1) \quad \text{ for all } v\in S\setminus\{u_1,u_2,w,w'\}.$$
Then $S\setminus\{u_1,u_2\}$ is an exact defensive \big[$\D_1-2(n_1-n)-2$\big]-alliance and this finishes the proof.
\end{proof}

\begin{theorem}\label{t:Reg4-6}
  Let $G$ be a connected $\D$-regular graph with $\D\le5$ and $G^*$ another graph. If $A(G^*;x) = A(G;x)$, then $G^*$ is a $\D$-regular graph with the same order and size of $G$. 
\end{theorem}

\begin{proof}
If $0\le\D\le3$, then the result follows from Theorem \ref{t:Reg0-3}. Assume that $4\le\D\le5$.
Let $n,n_1$ be the orders of $G,G^*$, respectively, and let $\d_1,\D_1$ be the minimum and maximum degree of $G^*$, respectively.
By Theorem \ref{t:RegCharact}, $G^*$ is a connected graph and $n_1\ge n$.
Seeking for a contradiction assume that $n_1> n$.

\medskip

Assume first $\D=4$.
By Theorem \ref{t:RegCharact} we have $n_1=n+\D_1-4$, $\D_1>4$ and $\D_1+\d_1\le 6$.
Thus, we have $\D_1=5$ and $\d_1=1$, and then $n_1=n+1$. Then, Theorem \ref{t:RegCharact} and Lemma \ref{l:d1=1} give that $A(G;x)=A(G^*;x)$ is not a monic polynomial of degree $n_1+3=n+4$.
This is the contradiction we were looking for, and we conclude $n_1=n$.

\medskip

Assume now $\D=5$.
By Theorem \ref{t:RegCharact} we have $n_1=n+\D_1-5$, $\D_1>5$, $\D_1+\d_1\le 8$, $\d_1+4\le\D_1$ and $\d_1\equiv \D_1 (\text{mod } 2)$.
Thus, we have the following cases:
\begin{description}
  \item[Case 1] {$\d_1=1$ and $\D_1=7$,}
  \item[Case 2] {$\d_1=2$ and $\D_1=6$.}
\end{description}
Lemma \ref{l:d1=1} gives that $A(G;x)$ is not a monic polynomial of degree $n+5$ in Case 1; this is the contradiction we were looking for, and we conclude $n_1=n$. In Case 2 we have $n_1=n+1$.
Since $A(G;x)$ is a monic polynomial of degree $n+5$, Lemma \ref{l:d1=1} gives that $n<10$.
Hence, Lemma \ref{l:An+D-2>n} gives that $A_{2}(G^*)> n$; however, Theorem \ref{t:An+D-2=n} gives $A_{3}(G)=n$.
This is the contradiction we were looking for, and we conclude $n_1=n$.
%
%
%
%
%
\end{proof}

\subsection{Computing the alliance polynomials for cubic graphs with small order}
\label{sect3}
In this subsection we compute the alliance polynomial of cubic graphs of small order by using Algorithm \ref{algorithm}, and find that non-isomorphic cubic graphs of order at most $10$ have different alliance polynomials. By Theorem \ref{t:Reg0-3} this implies these cubic graphs are uniquely determined by their alliance polynomial.
A similar study on characterization of cubic graphs with small order by their domination polynomials is done in \cite{AP1}, although it obtains a different result.

Computing the alliance polynomial of a graph $G$ on $n$ vertices and $m$ edges
by calculating $k_S$ for each connected induced subgraph $\langle S\rangle$ takes time $O(m2^n)$.
On $\D$-regular graphs the complexity is $O(n2^n)$.
Note that in order to decreasing this time for small size of $G$ could be used its topology by traveling each connected induced subgraph.
Testing whether $\langle S\rangle$ is connected can be done using Depth-First Search (DFS), and this has time complexity $O(m)$. Finding $k_S$ requires $O(n)$ time.

Let $G$ be a cubic graph with order $n$.
If $n=4$ then $G$ is isomorphic to $K_4$ and Theorem \ref{t:Reg0-3} gives uniqueness.
If $n=6$ then $G$ is isomorphic either to $K_{3,3}$ or to the Cartesian product $P_2\Box C_3$;
hence, Theorem \ref{t:Reg0-3} implies that they are uniquely determined by their alliance polynomial since $A(K_{3,3};x)=6 x^3 + 33 x^5 + 15 x^7 + x^9$ and $A(P_2\Box C_3;x)=6 x^3 + 33 x^5 + 11 x^7 + x^9$.
Notice that these alliance polynomials are equal except for the coefficient of $x^7$;
it is an interesting fact since many parameters of these graphs are different.


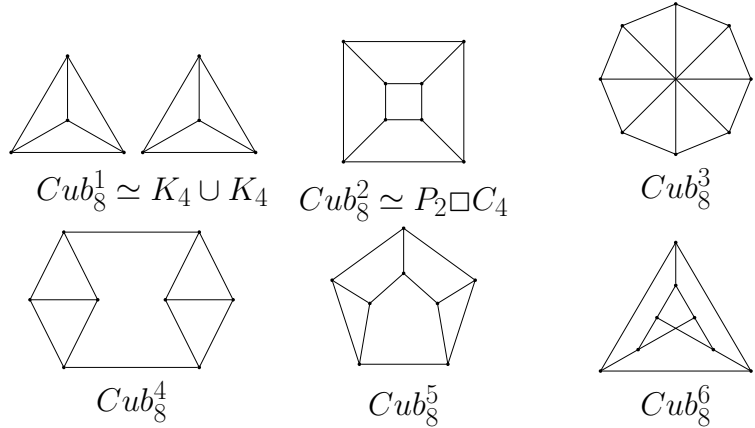
\begin{figure}[ht]
\centering
\begin{minipage}[t]{3.5cm}
    \centering
       \scalebox{.5}
       {\begin{pspicture}(-1.8,-1.05)(5.3,3)
       \psline[linewidth=0.01cm,dotsize=0.07055555cm 2.5]{*-*}(-1.5,0)(1.5,0)
       \psline[linewidth=0.01cm,dotsize=0.07055555cm 2.5]{*-*}(0,0.855)(0,2.565)
       \psline[linewidth=0.01cm,dotsize=0.07055555cm 2.5]{-}(0,2.565)(1.5,0)(0,0.855)(-1.5,0)(0,2.565)
       \psline[linewidth=0.01cm,dotsize=0.07055555cm 2.5]{*-*}(2,0)(5,0)
       \psline[linewidth=0.01cm,dotsize=0.07055555cm 2.5]{*-*}(3.5,0.855)(3.5,2.565)
       \psline[linewidth=0.01cm,dotsize=0.07055555cm 2.5]{-}(3.5,2.565)(5,0)(3.5,0.855)(2,0)(3.5,2.565)
       \uput[270](2.25,-0.3){\Huge{$Cub_8^1\simeq K_4\cup K_4$}}
       \end{pspicture}}
\end{minipage}
\begin{minipage}[t]{3.5cm}
    \centering
    \scalebox{.5}
    {\begin{pspicture}(-1.92,-2.4)(1.92,2.22)
    \psline[linewidth=0.01cm,dotsize=0.07055555cm 2.5]{*-*}(-1.6,-1.6)(1.6,-1.6)
    \psline[linewidth=0.01cm,dotsize=0.07055555cm 2.5]{*-*}(1.6,1.6)(-1.6,1.6)
    \psline[linewidth=0.01cm,dotsize=0.07055555cm 2.5]{*-*}(-0.48,0.48)(0.48,0.48)
    \psline[linewidth=0.01cm,dotsize=0.07055555cm 2.5]{*-*}(0.48,-0.48)(-0.48,-0.48)
    \psline[linewidth=0.01cm,dotsize=0.07055555cm 2.5]{-}(0.48,-0.48)(1.6,-1.6)(1.6,1.6)(0.48,0.48)(0.48,-0.48)
    \psline[linewidth=0.01cm,dotsize=0.07055555cm 2.5]{-}(-1.6,1.6)(-1.6,-1.6)(-0.48,-0.48)(-0.48,0.48)(-1.6,1.6)
    \uput[270](0,-1.92){\Huge{$Cub_8^2\simeq P_2\Box C_4$}}
    \end{pspicture}}
\end{minipage}
\begin{minipage}[t]{3.5cm}
      \centering
      \scalebox{.5}
      {\begin{pspicture}(-2.4,-3)(2.4,2.4)
      \psline[linewidth=0.01cm,dotsize=0.07055555cm 2.5]{*-*}(-2,0)(2,0)
      \psline[linewidth=0.01cm,dotsize=0.07055555cm 2.5]{*-*}(1.42,1.42)(-1.42,-1.42)
      \psline[linewidth=0.01cm,dotsize=0.07055555cm 2.5]{*-*}(0,-2)(0,2)
      \psline[linewidth=0.01cm,dotsize=0.07055555cm 2.5]{*-*}(-1.42,1.42)(1.42,-1.42)
      \psline[linewidth=0.01cm,dotsize=0.07055555cm 2.5]{-}(-2,0)(-1.42,-1.42)(0,-2)(1.42,-1.42)(2,0)(1.42,1.42)(0,2)(-1.42,1.42)(-2,0)
      \uput[270](0,-2.2){\Huge{$Cub_8^3$}}
      \end{pspicture}}
\end{minipage}

\vfill
\vspace{.35cm}

\centering
\begin{minipage}[t]{3.5cm}
    \centering
       \scalebox{.5}
       {\begin{pspicture}(-3,-2.5)(3,2)
       \psline[linewidth=0.01cm,dotsize=0.07055555cm 2.5]{*-*}(-1.8,-1.8)(1.8,-1.8)
       \psline[linewidth=0.01cm,dotsize=0.07055555cm 2.5]{*-*}(-2.7,0)(-0.9,0)
       \psline[linewidth=0.01cm,dotsize=0.07055555cm 2.5]{*-*}(0.9,0)(2.7,0)
       \psline[linewidth=0.01cm,dotsize=0.07055555cm 2.5]{*-*}(-1.8,1.8)(1.8,1.8)
       \psline[linewidth=0.01cm,dotsize=0.07055555cm 2.5]{-}(-1.8,1.8)(-2.7,0)(-1.8,-1.8)(-0.9,0)(-1.8,1.8)
       \psline[linewidth=0.01cm,dotsize=0.07055555cm 2.5]{-}(1.8,-1.8)(0.9,0)(1.8,1.8)(2.7,0)(1.8,-1.8)
       \uput[270](0,-2){\Huge{$Cub_8^4$}}
       \end{pspicture}}
\end{minipage}
\begin{minipage}[t]{3.5cm}
    \centering
       \scalebox{.5}
       {\begin{pspicture}(-2.2,-2.4)(2.2,2.2)
       \psline[linewidth=0.01cm,dotsize=0.07055555cm 2.5]{*-*}(-1.176,-1.618)(1.176,-1.618)
       \psline[linewidth=0.01cm,dotsize=0.07055555cm 2.5]{*-*}(1.902,0.618)(0,2)
       \psline[linewidth=0.01cm,dotsize=0.07055555cm 2.5]{*-*}(0,0.8)(0.9,0)
       \psline[linewidth=0.01cm,dotsize=0.07055555cm 2.5]{*-*}(-1.902,0.618)(-0.9,0)
       \psline[linewidth=0.01cm,dotsize=0.07055555cm 2.5]{-}(1.176,-1.618)(1.902,0.618)(0.9,0)(1.176,-1.618)
       \psline[linewidth=0.01cm,dotsize=0.07055555cm 2.5]{-}(0,2)(0,0.8)(-0.9,0)(-1.176,-1.618)(-1.902,0.618)(0,2)
       \uput[270](0,-2){\Huge{$Cub_8^5$}}
       \end{pspicture}}
\end{minipage}
\begin{minipage}[t]{3.5cm}
      \centering
       \scalebox{.5}
       {\begin{pspicture}(-2.4,-0.6)(2.4,3.5)
       \psline[linewidth=0.01cm,dotsize=0.07055555cm 2.5]{*-*}(-2,0)(2,0)
       \psline[linewidth=0.01cm,dotsize=0.07055555cm 2.5]{*-*}(0,2.28)(0,3.42)
       \psline[linewidth=0.01cm,dotsize=0.07055555cm 2.5]{*-*}(-1,0.57)(-0.5,1.42)
       \psline[linewidth=0.01cm,dotsize=0.07055555cm 2.5]{*-*}(1,0.57)(0.5,1.42)
       \psline[linewidth=0.01cm,dotsize=0.07055555cm 2.5]{-}(-2,0)(0,3.42)(2,0)(1,0.57)(-0.5,1.42)(0,2.28)(0.5,1.42)(-1,0.57)(-2,0)
       \uput[270](0,-0.2){\Huge{$Cub_8^6$}}
       \end{pspicture}}
\end{minipage}
\caption{Cubic graphs with order $8$.}
\label{fig:ord8}
\end{figure}


\begin{table}[ht]
  \begin{center}
  \scalebox{.85}{
    \begin{tabular}{|c|c|c|c|}
        \hline
         &&&\\
        Graph & Alliance polynomial & Graph & Alliance polynomial \\
         &&&\\
        \hline
         &&&\\
         $Cub_8^1$ & $8 x^5 + 12 x^7 + 8 x^9 + 2x^{11}$ & $Cub_8^4$ & $8 x^5 + 94 x^7 + 20 x^9 + x^{11}$ \\
         &&&\\
        \hline
         &&&\\
         $Cub_8^2$ & $8 x^5 + 128 x^7 + 30 x^9 + x^{11}$ & $Cub_8^5$ & $8 x^5 + 118 x^7 + 24 x^9 + x^{11}$ \\
         &&&\\
        \hline
         &&&\\
         $Cub_8^3$ & $8 x^5 + 132 x^7 + 32 x^9 + x^{11}$ & $Cub_8^6$ & $8 x^5 + 126 x^7 + 28 x^9 + x^{11}$ \\
         &&&\\
        \hline
    \end{tabular}}
    \vspace{.35cm}
  \end{center}
  \caption{Alliance polynomials of cubic graph of order $8$.}
  \label{tab:ord8}
\end{table}

Figure \ref{fig:ord8} shows the cubic graphs with order $8$ and Table \ref{tab:ord8} their alliance polynomials; since they are different, Theorem \ref{t:Reg0-3} gives their uniqueness.



\begin{figure}[ht]
\centering
\begin{minipage}[t]{3.5cm}
    \centering
       \scalebox{.5}
       {\begin{pspicture}(-2.2,-2.4)(2.2,2.2)
       \psline[linewidth=0.01cm,dotsize=0.07055555cm 2.5]{*-*}(-1.176,-1.618)(-0.588,-0.809)
       \psline[linewidth=0.01cm,dotsize=0.07055555cm 2.5]{*-*}(1.176,-1.618)(0.588,-0.809)
       \psline[linewidth=0.01cm,dotsize=0.07055555cm 2.5]{*-*}(0,2)(0,1)
       \psline[linewidth=0.01cm,dotsize=0.07055555cm 2.5]{*-*}(-1.902,0.618)(-0.951,0.309)
       \psline[linewidth=0.01cm,dotsize=0.07055555cm 2.5]{*-*}(1.902,0.618)(0.951,0.309)
       \psline[linewidth=0.01cm,dotsize=0.07055555cm 2.5]{-}(1.176,-1.618)(1.902,0.618)(0,2)(-1.902,0.618)(-1.176,-1.618)(1.176,-1.618)
       \psline[linewidth=0.01cm,dotsize=0.07055555cm 2.5]{-}(0,1)(-0.588,-0.809)(0.951,0.309)(-0.951,0.309)(0.588,-0.809)(0,1)
       \uput[270](0,-2){\huge{$Cub_{10}^1$}}
       \end{pspicture}}
\end{minipage}
\begin{minipage}[t]{3.5cm}
    \centering
       \scalebox{.5}
       {\begin{pspicture}(-2.2,-2.4)(2.2,2.2)
       \psline[linewidth=0.01cm,dotsize=0.07055555cm 2.5]{*-*}(-1.176,-1.618)(-0.588,-0.809)
       \psline[linewidth=0.01cm,dotsize=0.07055555cm 2.5]{*-*}(1.176,-1.618)(0.588,-0.809)
       \psline[linewidth=0.01cm,dotsize=0.07055555cm 2.5]{*-*}(0,2)(0,1)
       \psline[linewidth=0.01cm,dotsize=0.07055555cm 2.5]{*-*}(-1.902,0.618)(-0.951,0.309)
       \psline[linewidth=0.01cm,dotsize=0.07055555cm 2.5]{*-*}(1.902,0.618)(0.951,0.309)
       \psline[linewidth=0.01cm,dotsize=0.07055555cm 2.5]{-}(-1.176,-1.618)(1.176,-1.618)(1.902,0.618)(0,2)(-1.902,0.618)(-1.176,-1.618)
       \psline[linewidth=0.01cm,dotsize=0.07055555cm 2.5]{-}(0.588,-0.809)(0.951,0.309)(0,1)(-0.951,0.309)(-0.588,-0.809)(0.588,-0.809)
       \uput[270](0,-2){\huge{$Cub_{10}^2$}}
       \end{pspicture}}
\end{minipage}
\begin{minipage}[t]{3.5cm}
      \centering
       \scalebox{.5}
       {\begin{pspicture}(-2.2,-2.4)(2.2,2.2)
       \psline[linewidth=0.01cm,dotsize=0.07055555cm 2.5]{*-*}(-1.176,-1.618)(-0.588,-0.809)
       \psline[linewidth=0.01cm,dotsize=0.07055555cm 2.5]{*-*}(1.176,-1.618)(0.588,-0.809)
       \psline[linewidth=0.01cm,dotsize=0.07055555cm 2.5]{*-*}(0,2)(0,1)
       \psline[linewidth=0.01cm,dotsize=0.07055555cm 2.5]{*-*}(-1.902,0.618)(-0.951,0.309)
       \psline[linewidth=0.01cm,dotsize=0.07055555cm 2.5]{*-*}(1.902,0.618)(0.951,0.309)
       \psline[linewidth=0.01cm,dotsize=0.07055555cm 2.5]{-}(-0.588,-0.809)(1.176,-1.618)(1.902,0.618)(0,2)(-1.902,0.618)(-1.176,-1.618)(0.588,-0.809) 
       \psline[linewidth=0.01cm,dotsize=0.07055555cm 2.5]{-}(0.588,-0.809)(0.951,0.309)(0,1)(-0.951,0.309)(-0.588,-0.809)
       \uput[270](0,-2){\huge{$Cub_{10}^3$}}
       \end{pspicture}}
\end{minipage}
\begin{minipage}[t]{3.5cm}
      \centering
       \scalebox{.5}
       {\begin{pspicture}(-2.2,-2.4)(2.2,2.2)
       \psline[linewidth=0.01cm,dotsize=0.07055555cm 2.5]{*-*}(-1.176,-1.618)(0.588,-0.809)
       \psline[linewidth=0.01cm,dotsize=0.07055555cm 2.5]{*-*}(1.176,-1.618)(-0.588,-0.809)
       \psline[linewidth=0.01cm,dotsize=0.07055555cm 2.5]{*-*}(0,2)(0,1)
       \psline[linewidth=0.01cm,dotsize=0.07055555cm 2.5]{*-*}(-1.902,0.618)(-0.951,0.309)
       \psline[linewidth=0.01cm,dotsize=0.07055555cm 2.5]{*-*}(1.902,0.618)(0.951,0.309)
       \psline[linewidth=0.01cm,dotsize=0.07055555cm 2.5]{-}(-1.176,-1.618)(1.176,-1.618)(1.902,0.618)(0,2)(-1.902,0.618)(-1.176,-1.618)
       \psline[linewidth=0.01cm,dotsize=0.07055555cm 2.5]{-}(0.588,-0.809)(0.951,0.309)(0,1)(-0.951,0.309)(-0.588,-0.809)(0.588,-0.809)
       \uput[270](0,-2){\huge{$Cub_{10}^4$}}
       \end{pspicture}}
\end{minipage}
\vfill
\vspace{.4cm}
\centering
\begin{minipage}[t]{3.5cm}
    \centering
       \scalebox{.5}
       {\begin{pspicture}(-2.2,-2.4)(2.2,2.2)
       \psline[linewidth=0.01cm,dotsize=0.07055555cm 2.5]{*-*}(-1.176,-1.618)(1.176,-1.618)
       \psline[linewidth=0.01cm,dotsize=0.07055555cm 2.5]{*-*}(0,-1)(0,-0.4)
       \psline[linewidth=0.01cm,dotsize=0.07055555cm 2.5]{*-*}(0,2)(0,1)
       \psline[linewidth=0.01cm,dotsize=0.07055555cm 2.5]{*-*}(-1.902,0.618)(-1,0)
       \psline[linewidth=0.01cm,dotsize=0.07055555cm 2.5]{*-*}(1.902,0.618)(1,0)
       \psline[linewidth=0.01cm,dotsize=0.07055555cm 2.5]{-}(-1.176,-1.618)(0,-1)(1.176,-1.618)(1.902,0.618)(0,2)(-1.902,0.618)(-1.176,-1.618)
       \psline[linewidth=0.01cm,dotsize=0.07055555cm 2.5]{-}(0,1)(1,0)(0,-0.4)(-1,0)(0,1)
       \uput[270](0,-2){\huge{$Cub_{10}^5$}}
       \end{pspicture}}
\end{minipage}
\begin{minipage}[t]{3.5cm}
    \centering
       \scalebox{.5}
       {\begin{pspicture}(-2.2,-2.4)(2.2,2.2)
       \psline[linewidth=0.01cm,dotsize=0.07055555cm 2.5]{*-*}(-2,0)(2,0)
       \psline[linewidth=0.01cm,dotsize=0.07055555cm 2.5]{*-*}(1.618,1.176)(0.618,-1.902)
       \psline[linewidth=0.01cm,dotsize=0.07055555cm 2.5]{*-*}(0.618,1.902)(-1.618,-1.176)
       \psline[linewidth=0.01cm,dotsize=0.07055555cm 2.5]{*-*}(-0.618,1.902)(1.618,-1.176)
       \psline[linewidth=0.01cm,dotsize=0.07055555cm 2.5]{*-*}(-1.618,1.176)(-0.618,-1.902)
       \psline[linewidth=0.01cm,dotsize=0.07055555cm 2.5]{-}(2,0)(1.618,1.176)(0.618,1.902)(-0.618,1.902)(-1.618,1.176)(-2,0)(-1.618,-1.176)(-0.618,-1.902)(0.618,-1.902)(1.618,-1.176)(2,0)
       \uput[270](0,-2){\huge{$Cub_{10}^6$}}
       \end{pspicture}}
\end{minipage}
\begin{minipage}[t]{3.5cm}
      \centering
       \scalebox{.5}
       {\begin{pspicture}(-2.2,-2.4)(2.2,2.2)
       \psline[linewidth=0.01cm,dotsize=0.07055555cm 2.5]{*-*}(2,0)(-0.618,1.902)
       \psline[linewidth=0.01cm,dotsize=0.07055555cm 2.5]{*-*}(-2,0)(0.618,-1.902)
       \psline[linewidth=0.01cm,dotsize=0.07055555cm 2.5]{*-*}(0.618,1.902)(-1.618,-1.176)
       \psline[linewidth=0.01cm,dotsize=0.07055555cm 2.5]{*-*}(-0.618,-1.902)(1.618,-1.176)
       \psline[linewidth=0.01cm,dotsize=0.07055555cm 2.5]{*-*}(-1.618,1.176)(1.618,1.176)
       \psline[linewidth=0.01cm,dotsize=0.07055555cm 2.5]{-}(2,0)(1.618,1.176)(0.618,1.902)(-0.618,1.902)(-1.618,1.176)(-2,0)(-1.618,-1.176)(-0.618,-1.902)(0.618,-1.902)(1.618,-1.176)(2,0)
       \uput[270](0,-2){\huge{$Cub_{10}^7$}}
       \end{pspicture}}
\end{minipage}
\begin{minipage}[t]{3.5cm}
      \centering
       \scalebox{.5}
       {\begin{pspicture}(-2.2,-2.4)(2.2,2.2)
       \psline[linewidth=0.01cm,dotsize=0.07055555cm 2.5]{*-*}(-2,0)(2,0)
       \psline[linewidth=0.01cm,dotsize=0.07055555cm 2.5]{*-*}(1.618,1.176)(-0.618,1.902)
       \psline[linewidth=0.01cm,dotsize=0.07055555cm 2.5]{*-*}(0.618,1.902)(-1.618,-1.176)
       \psline[linewidth=0.01cm,dotsize=0.07055555cm 2.5]{*-*}(0.618,-1.902)(-1.618,1.176)
       \psline[linewidth=0.01cm,dotsize=0.07055555cm 2.5]{*-*}(1.618,-1.176)(-0.618,-1.902)
       \psline[linewidth=0.01cm,dotsize=0.07055555cm 2.5]{-}(2,0)(1.618,1.176)(0.618,1.902)(-0.618,1.902)(-1.618,1.176)(-2,0)(-1.618,-1.176)(-0.618,-1.902)(0.618,-1.902)(1.618,-1.176)(2,0)
       \uput[270](0,-2){\huge{$Cub_{10}^8$}}
       \end{pspicture}}
\end{minipage}
\vfill
\vspace{.4cm}
\centering
\begin{minipage}[t]{3.5cm}
    \centering
       \scalebox{.5}
       {\begin{pspicture}(-2.2,-2.4)(2.2,2.2)
       \psline[linewidth=0.01cm,dotsize=0.07055555cm 2.5]{*-*}(-2,0)(2,0)
       \psline[linewidth=0.01cm,dotsize=0.07055555cm 2.5]{*-*}(1.618,1.176)(1.618,-1.176)
       \psline[linewidth=0.01cm,dotsize=0.07055555cm 2.5]{*-*}(0.618,1.902)(0.618,-1.902)
       \psline[linewidth=0.01cm,dotsize=0.07055555cm 2.5]{*-*}(-0.618,1.902)(-0.618,-1.902)
       \psline[linewidth=0.01cm,dotsize=0.07055555cm 2.5]{*-*}(-1.618,1.176)(-1.618,-1.176)
       \psline[linewidth=0.01cm,dotsize=0.07055555cm 2.5]{-}(2,0)(1.618,1.176)(0.618,1.902)(-0.618,1.902)(-1.618,1.176)(-2,0)(-1.618,-1.176)(-0.618,-1.902)(0.618,-1.902)(1.618,-1.176)(2,0)
       \uput[270](0,-2){\huge{$Cub_{10}^9$}}
       \end{pspicture}}
\end{minipage}
\begin{minipage}[t]{3.5cm}
    \centering
       \scalebox{.5}
       {\begin{pspicture}(-2.2,-2.4)(2.2,2.2)
       \psline[linewidth=0.01cm,dotsize=0.07055555cm 2.5]{*-*}(-2,0)(2,0)
       \psline[linewidth=0.01cm,dotsize=0.07055555cm 2.5]{*-*}(1.618,1.176)(1.618,-1.176)
       \psline[linewidth=0.01cm,dotsize=0.07055555cm 2.5]{*-*}(0.618,1.902)(-0.618,-1.902)
       \psline[linewidth=0.01cm,dotsize=0.07055555cm 2.5]{*-*}(-0.618,1.902)(0.618,-1.902)
       \psline[linewidth=0.01cm,dotsize=0.07055555cm 2.5]{*-*}(-1.618,1.176)(-1.618,-1.176)
       \psline[linewidth=0.01cm,dotsize=0.07055555cm 2.5]{-}(2,0)(1.618,1.176)(0.618,1.902)(-0.618,1.902)(-1.618,1.176)(-2,0)(-1.618,-1.176)(-0.618,-1.902)(0.618,-1.902)(1.618,-1.176)(2,0)
       \uput[270](0,-2){\huge{$Cub_{10}^{10}$}}
       \end{pspicture}}
\end{minipage}
\begin{minipage}[t]{3.5cm}
      \centering
       \scalebox{.5}
       {\begin{pspicture}(-2.2,-2.4)(2.2,2.2)
       \psline[linewidth=0.01cm,dotsize=0.07055555cm 2.5]{*-*}(2,0)(-0.618,-1.902)
       \psline[linewidth=0.01cm,dotsize=0.07055555cm 2.5]{*-*}(1.618,1.176)(-0.618,1.902)
       \psline[linewidth=0.01cm,dotsize=0.07055555cm 2.5]{*-*}(0.618,1.902)(-1.618,1.176)
       \psline[linewidth=0.01cm,dotsize=0.07055555cm 2.5]{*-*}(-2,0)(0.618,-1.902)
       \psline[linewidth=0.01cm,dotsize=0.07055555cm 2.5]{*-*}(-1.618,-1.176)(1.618,-1.176)
       \psline[linewidth=0.01cm,dotsize=0.07055555cm 2.5]{-}(2,0)(1.618,1.176)(0.618,1.902)(-0.618,1.902)(-1.618,1.176)(-2,0)(-1.618,-1.176)(-0.618,-1.902)(0.618,-1.902)(1.618,-1.176)(2,0)
       \uput[270](0,-2){\huge{$Cub_{10}^{11}$}}
       \end{pspicture}}
\end{minipage}
\begin{minipage}[t]{3.5cm}
      \centering
       \scalebox{.5}
       {\begin{pspicture}(-2.2,-2.4)(2.2,2.2)
       \psline[linewidth=0.01cm,dotsize=0.07055555cm 2.5]{*-*}(2,0)(-1.618,-1.176)
       \psline[linewidth=0.01cm,dotsize=0.07055555cm 2.5]{*-*}(1.618,1.176)(-0.618,1.902)
       \psline[linewidth=0.01cm,dotsize=0.07055555cm 2.5]{*-*}(0.618,1.902)(-1.618,1.176)
       \psline[linewidth=0.01cm,dotsize=0.07055555cm 2.5]{*-*}(-2,0)(0.618,-1.902)
       \psline[linewidth=0.01cm,dotsize=0.07055555cm 2.5]{*-*}(1.618,-1.176)(-0.618,-1.902)
       \psline[linewidth=0.01cm,dotsize=0.07055555cm 2.5]{-}(2,0)(1.618,1.176)(0.618,1.902)(-0.618,1.902)(-1.618,1.176)(-2,0)(-1.618,-1.176)(-0.618,-1.902)(0.618,-1.902)(1.618,-1.176)(2,0)
       \uput[270](0,-2){\huge{$Cub_{10}^{12}$}}
       \end{pspicture}}
\end{minipage}
\vfill
\vspace{.4cm}
\centering
\begin{minipage}[t]{3.5cm}
    \centering
       \scalebox{.5}
       {\begin{pspicture}(-2.2,-2.4)(2.2,2.2)
       \psline[linewidth=0.01cm,dotsize=0.07055555cm 2.5]{*-*}(2,0)(0.618,-1.902)
       \psline[linewidth=0.01cm,dotsize=0.07055555cm 2.5]{*-*}(1.618,1.176)(-0.618,1.902)
       \psline[linewidth=0.01cm,dotsize=0.07055555cm 2.5]{*-*}(-1.618,1.176)(0.618,1.902)
       \psline[linewidth=0.01cm,dotsize=0.07055555cm 2.5]{*-*}(-2,0)(-0.618,-1.902)
       \psline[linewidth=0.01cm,dotsize=0.07055555cm 2.5]{*-*}(-1.618,-1.176)(1.618,-1.176)
       \psline[linewidth=0.01cm,dotsize=0.07055555cm 2.5]{-}(2,0)(1.618,1.176)(0.618,1.902)(-0.618,1.902)(-1.618,1.176)(-2,0)(-1.618,-1.176)(-0.618,-1.902)(0.618,-1.902)(1.618,-1.176)(2,0)
       \uput[270](0,-2){\huge{$Cub_{10}^{13}$}}
       \end{pspicture}}
\end{minipage}
\begin{minipage}[t]{3.5cm}
    \centering
       \scalebox{.5}
       {\begin{pspicture}(-2.2,-2.4)(2.2,2.2)
       \psline[linewidth=0.01cm,dotsize=0.07055555cm 2.5]{*-*}(-2,0)(2,0)
       \psline[linewidth=0.01cm,dotsize=0.07055555cm 2.5]{*-*}(1.618,1.176)(0.618,-1.902)
       \psline[linewidth=0.01cm,dotsize=0.07055555cm 2.5]{*-*}(0.618,1.902)(1.618,-1.176)
       \psline[linewidth=0.01cm,dotsize=0.07055555cm 2.5]{*-*}(-1.618,1.176)(-0.618,-1.902)
       \psline[linewidth=0.01cm,dotsize=0.07055555cm 2.5]{*-*}(-0.618,1.902)(-1.618,-1.176)
       \psline[linewidth=0.01cm,dotsize=0.07055555cm 2.5]{-}(2,0)(1.618,1.176)(0.618,1.902)(-0.618,1.902)(-1.618,1.176)(-2,0)(-1.618,-1.176)(-0.618,-1.902)(0.618,-1.902)(1.618,-1.176)(2,0)
       \uput[270](0,-2){\huge{$Cub_{10}^{14}$}}
       \end{pspicture}}
\end{minipage}
\begin{minipage}[t]{3.5cm}
      \centering
      \scalebox{.5}
      {\begin{pspicture}(-2.2,-2.4)(2.2,2.2)
      \psline[linewidth=0.01cm,dotsize=0.07055555cm 2.5]{*-*}(-2,0)(2,0)
      \psline[linewidth=0.01cm,dotsize=0.07055555cm 2.5]{*-*}(0,-0.7)(0,-2)
      \psline[linewidth=0.01cm,dotsize=0.07055555cm 2.5]{*-*}(1.42,1.42)(0,0.7)(-1.42,1.42)
      \psline[linewidth=0.01cm,dotsize=0.07055555cm 2.5]{*-*}(0,0.7)(0,2)
      \psline[linewidth=0.01cm,dotsize=0.07055555cm 2.5]{*-*}(-1.42,-1.42)(0,-0.7)(1.42,-1.42)
      \psline[linewidth=0.01cm,dotsize=0.07055555cm 2.5]{-}(-2,0)(-1.42,-1.42)(0,-2)(1.42,-1.42)(2,0)(1.42,1.42)(0,2)(-1.42,1.42)(-2,0)
      \uput[270](0,-2){\huge{$Cub_{10}^{15}$}}
      \end{pspicture}}
\end{minipage}
\begin{minipage}[t]{3.5cm}
      \centering
      \scalebox{.5}
      {\begin{pspicture}(-2.2,-2.4)(2.2,2.2)
      \psline[linewidth=0.01cm,dotsize=0.07055555cm 2.5]{*-*}(0,0)(0.47,1.14)
      \psline[linewidth=0.01cm,dotsize=0.07055555cm 2.5]{*-*}(2,0)(-1.42,-1.42)
      \psline[linewidth=0.01cm,dotsize=0.07055555cm 2.5]{*-*}(1.42,1.42)(0.47,1.14)(0,2)
      \psline[linewidth=0.01cm,dotsize=0.07055555cm 2.5]{*-*}(0,-2)(-1.42,1.42)
      \psline[linewidth=0.01cm,dotsize=0.07055555cm 2.5]{*-*}(-2,0)(0,0)(1.42,-1.42)
      \psline[linewidth=0.01cm,dotsize=0.07055555cm 2.5]{-}(-2,0)(-1.42,-1.42)(0,-2)(1.42,-1.42)(2,0)(1.42,1.42)(0,2)(-1.42,1.42)(-2,0)
      \uput[270](0,-2){\huge{$Cub_{10}^{16}$}}
      \end{pspicture}}
\end{minipage}
\vfill
\vspace{.4cm}
\centering
\begin{minipage}[t]{3.5cm}
    \centering
      \scalebox{.5}
      {\begin{pspicture}(-2.2,-1.4)(2.2,3.2)
      \psline[linewidth=0.01cm,dotsize=0.07055555cm 2.5]{*-*}(0,3)(0.667,2)
      \psline[linewidth=0.01cm,dotsize=0.07055555cm 2.5]{*-*}(0,-1)(-0.667,0)
      \psline[linewidth=0.01cm,dotsize=0.07055555cm 2.5]{*-*}(-2,2)(0,3)(2,2)
      \psline[linewidth=0.01cm,dotsize=0.07055555cm 2.5]{*-*}(-0.667,2)(0.667,0)
      \psline[linewidth=0.01cm,dotsize=0.07055555cm 2.5]{*-*}(-2,0)(0,-1)(2,0)
      \psline[linewidth=0.01cm,dotsize=0.07055555cm 2.5]{-}(-2,0)(-0.667,0)(0.667,0)(2,0)(2,2)(0.667,2)(-0.667,2)(-2,2)(-2,0)
      \uput[270](0,-1){\huge{$Cub_{10}^{17}$}}
      \end{pspicture}}
\end{minipage}
\begin{minipage}[t]{3.5cm}
    \centering
      \scalebox{.5}
      {\begin{pspicture}(-2.2,-2.4)(2.2,2.2)
      \psline[linewidth=0.01cm,dotsize=0.07055555cm 2.5]{*-*}(0,0)(0,1.15)
      \psline[linewidth=0.01cm,dotsize=0.07055555cm 2.5]{*-*}(1,1.73)(0,1.15)(-1,1.73)
      \psline[linewidth=0.01cm,dotsize=0.07055555cm 2.5]{*-*}(-0.86,-0.5)(0,0)(0.86,-0.5)
      \psline[linewidth=0.01cm,dotsize=0.07055555cm 2.5]{*-*}(2,0)(0.86,-0.5)(1,-1.73)
      \psline[linewidth=0.01cm,dotsize=0.07055555cm 2.5]{*-*}(-2,0)(-0.86,-0.5)(-1,-1.73)
      \psline[linewidth=0.01cm,dotsize=0.07055555cm 2.5]{-}(-2,0)(-1,-1.73)(1,-1.73)(2,0)(1,1.73)(-1,1.73)(-2,0)
      \uput[270](0,-2){\huge{$Cub_{10}^{18}$}}
      \end{pspicture}}
\end{minipage}
\begin{minipage}[t]{3.5cm}
      \centering
      \scalebox{.5}
      {\begin{pspicture}(-2.7,-2.4)(2.7,2.2)
      \psline[linewidth=0.01cm,dotsize=0.07055555cm 2.5]{*-*}(-0.75,0)(0.75,0)
      \psline[linewidth=0.01cm,dotsize=0.07055555cm 2.5]{*-*}(-1.7865,1.4265)(-3.4635,-0.882)
      \psline[linewidth=0.01cm,dotsize=0.07055555cm 2.5]{*-*}(-1.7865,-1.4265)(-3.4635,0.882)
      \psline[linewidth=0.01cm,dotsize=0.07055555cm 2.5]{*-*}(1.7865,1.4265)(3.4635,-0.882)
      \psline[linewidth=0.01cm,dotsize=0.07055555cm 2.5]{*-*}(1.7865,-1.4265)(3.4635,0.882)
      \psline[linewidth=0.01cm,dotsize=0.07055555cm 2.5]{-}(-0.75,0)(-1.7865,1.4265)(-3.4635,0.882)(-3.4635,-0.882)(-1.7865,-1.4265)(-0.75,0)
      \psline[linewidth=0.01cm,dotsize=0.07055555cm 2.5]{-}(0.75,0)(1.7865,1.4265)(3.4635,0.882)(3.4635,-0.882)(1.7865,-1.4265)(0.75,0)
      \uput[270](0,-2){\huge{$Cub_{10}^{19}$}}
      \end{pspicture}}
\end{minipage}
\begin{minipage}[t]{3.5cm}
      \scalebox{.5}
      {\begin{pspicture}(-2.7,-2.4)(2.7,2.2)
      \uput[270](0.2,1.1){\huge{$Cub_{10}^{20}\simeq K_4 \cup K_{3,3}$}}
      \uput[270](0.2,-1.1){\huge{$Cub_{10}^{21}\simeq K_4 \cup P_2\Box C_3$}}
      \end{pspicture}}
\end{minipage}
\caption{Cubic graphs with order $10$.}
\label{fig:ord10}
\end{figure}
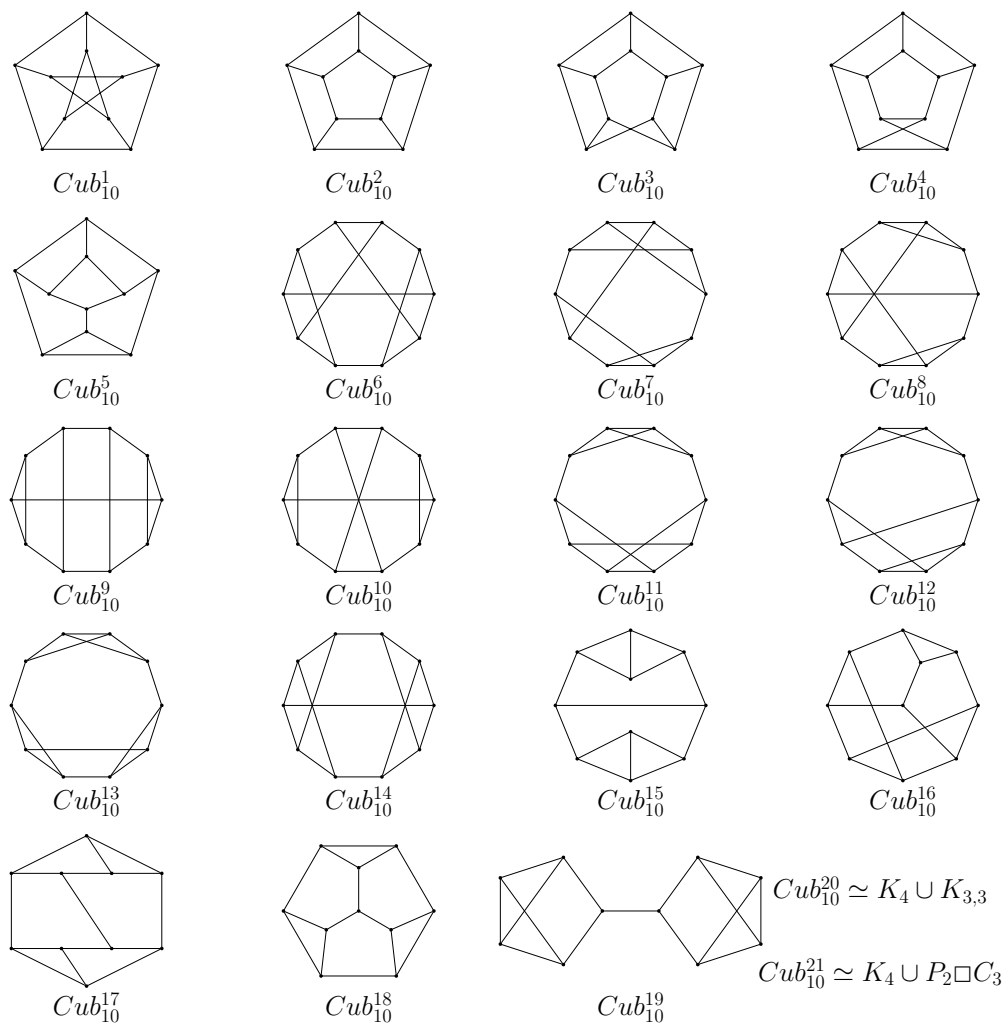

\begin{table}[ht]
   \centering
   \scalebox{.85}{
    \begin{tabular}{|c|c|c|c|c|c|}
        \hline
         &&&&&\\
        Graph & Alliance polynomial & Graph & Alliance polynomial & Graph & Alliance polynomial\\
         &&&&&\\
        \hline
         &&&&&\\
         $Cub_{10}^1$ & $10 x^7 + 480 x^9 + 77 x^{11} + x^{13}$ & $Cub_{10}^{8}$ & $10 x^7 + 407 x^9 + 56 x^{11} + x^{13}$ & $Cub_{10}^{15}$ & $10 x^7 + 272 x^9 + 42 x^{11} + x^{13}$ \\
         &&&&&\\
        \hline
         &&&&&\\
         $Cub_{10}^2$ & $10 x^7 + 425 x^9 + 67 x^{11} + x^{13}$ & $Cub_{10}^{9}$ & $10 x^7 + 357 x^9 + 53 x^{11} + x^{13}$ & $Cub_{10}^{16}$ & $10 x^7 + 419 x^9 + 62 x^{11} + x^{13}$ \\
         &&&&&\\
        \hline
         &&&&&\\
         $Cub_{10}^3$ & $10 x^7 + 435 x^9 + 65 x^{11} + x^{13}$ & $Cub_{10}^{10}$ & $10 x^7 + 387 x^9 + 55 x^{11} + x^{13}$ & $Cub_{10}^{17}$ & $10 x^7 + 372 x^9 + 54 x^{11} + x^{13}$ \\
         &&&&&\\
        \hline
         &&&&&\\
         $Cub_{10}^4$ & $10 x^7 + 451 x^9 + 69 x^{11} + x^{13}$ & $Cub_{10}^{11}$ & $10 x^7 + 307 x^9 + 55 x^{11} + x^{13}$ & $Cub_{10}^{18}$ & $10 x^7 + 351 x^9 + 50 x^{11} + x^{13}$ \\
         &&&&&\\
        \hline
         &&&&&\\
         $Cub_{10}^5$ & $10 x^7 + 404 x^9 + 61 x^{11} + x^{13}$ & $Cub_{10}^{12}$ & $10 x^7 + 304 x^9 + 48 x^{11} + x^{13}$ & $Cub_{10}^{19}$ & $10 x^7 + 176 x^9 + 36 x^{11} + x^{13}$ \\
         &&&&&\\
        \hline
         &&&&&\\
         $Cub_{10}^6$ & $10 x^7 + 462 x^9 + 67 x^{11} + x^{13}$ & $Cub_{10}^{13}$ & $10 x^7 + 267 x^9 + 43 x^{11} + x^{13}$ & $Cub_{10}^{20}$ & $10 x^7 + 39 x^9 + 19 x^{11} + 2 x^{13}$ \\
         &&&&&\\
        \hline
         &&&&&\\
         $Cub_{10}^7$ & $10 x^7 + 393 x^9 + 61 x^{11} + x^{13}$ & $Cub_{10}^{14}$ & $10 x^7 + 424 x^9 + 67 x^{11} + x^{13}$ & $Cub_{10}^{21}$ & $10 x^7 + 39 x^9 + 15 x^{11} + 2 x^{13}$ \\
         &&&&&\\
        \hline
    \end{tabular}}
    \vspace{.4cm}
  \caption{Alliance polynomials of cubic graph of order $10$.}
  \label{tab:ord10}
\end{table}

Figure \ref{fig:ord10} shows the cubic graphs with order $10$ and Table \ref{tab:ord10} their alliance polynomials. Since they are different, Theorem \ref{t:Reg0-3} gives their uniqueness.

We say that a graph $G$ is characterized by a graph polynomial $f$ if for every graph $G'$ such that $f(G') = f(G)$ we have that $G'$ is isomorphic to $G$.
A set of graphs $K$ is characterized by a graph polynomial $f$ if every graph $G \in K$ is characterized by $f$.

\begin{proposition}\label{p:uniqCub10}
Every cubic graph of order at most $10$ is characterized by its alliance polynomial.
\end{proposition}

Particularly, by Theorem \ref{p:AlliPoly} ii) we have that the cubic graphs of order at most $10$ are characterized by the evaluation at $x=1$ of their alliance polynomials.

\begin{proposition}\label{p:Eval1Cub10}
Two non-isomorphic cubic graphs of order at most $10$ have a different number of connected induced subgraphs.
\end{proposition}

Here we have proved that a polynomial cannot be the alliance polynomial of both a cubic and a non-cubic graph.

In \cite{CRST} the authors prove that paths, cycles, complete graphs, complete graphs minus an edge, and stars are all characterized by their alliance polynomials.
Here we have proved that a polynomial cannot be the alliance polynomial of both a cubic and a non-cubic graph.
In fact, cubic graphs of order at most $10$ are characterized by their alliance polynomials (Proposition \ref{p:uniqCub10}).

In \cite[Proposition 4.1]{CRST} the authors compare the distinctive power of the alliance polynomial with other well-known graph polynomials, such as the domination polynomial \cite{AP}, the independence polynomial \cite{GH}, the matching polynomial \cite{F}, the characteristic polynomial, the Tutte polynomial \cite{T}, the bivariate chromatic polynomial \cite{DPT} and the subgraph component polynomial \cite{TAM}.
In fact, this result exhibits for each of these polynomials $p(G;x)$ two graphs $G_1,G_2$ with
$p(G_1;x) = p(G_2;x)$ and $A(G_1;x) \neq A(G_2;x)$.

\section*{Acknowledgements}
This work was partly supported by a grant for Mobility of own research program at the University Carlos III de Madrid, a grant from Ministerio de Econom{\'\i}a y Competitividad (MTM 2013-46374-P) Spain, and a grant from CONACYT \ (CONACYT-UAG I0110/62/10), Mexico.

\

\end{document}